\documentclass[oneside,a4paper]{amsart}

\usepackage[T1]{fontenc}
\usepackage[english]{babel}

\usepackage{amsmath,amssymb,amscd,mathrsfs,epic,empheq}

\usepackage[colorlinks=true,hyperindex, linkcolor=magenta, pagebackref=false, citecolor=cyan,pdfpagelabels]{hyperref}
\usepackage[all]{xy}
\usepackage{mathabx}
\usepackage{placeins} 
\usepackage[usenames, dvipsnames]{xcolor} 
\usepackage{euscript} 
\usepackage{bbm} 
\usepackage{enumitem} 
\usepackage[nameinlink,capitalise,noabbrev]{cleveref} 
\usepackage{todonotes}

\hypersetup{
    colorlinks=true,
    linkcolor=PineGreen,
    filecolor=magenta,      
    citecolor=cyan,
}

\usepackage{euscript}
\usepackage{graphicx}
\usepackage{stmaryrd}

\usepackage[margin=1in]{geometry}

\usepackage{tikz}
\usetikzlibrary{arrows,backgrounds,
    decorations.pathreplacing,
    decorations.pathmorphing
}
\usetikzlibrary{matrix,arrows}

\numberwithin{equation}{subsection}

\newtheorem{Proposition}{Proposition}[section]
\newtheorem{Lemma}[Proposition]{Lemma}
\newtheorem{Theorem}[Proposition]{Theorem}
\newtheorem{Claim}[Proposition]{Claim}
\newtheorem{Corollary}[Proposition]{Corollary}

\theoremstyle{definition}
\newtheorem{Definition}[Proposition]{Definition}

\newtheorem{Construction}[Proposition]{Construction}
\newtheorem{Remark}[Proposition]{Remark}

\newtheorem{Notation}[Proposition]{Notation}
\newtheorem{Example}[Proposition]{Example}

\newcommand{\introthmname}{}
\newtheorem{introthminn}{\introthmname}
\newenvironment{introthm}[1]
  {\renewcommand{\introthmname}{#1}\begin{introthminn}}
  {\end{introthminn}}

\newcommand{\euXym}{\ensuremath{\euXymatrix@1}}

\newcommand{\euP}{\ensuremath{\EuScript{P}}}

\newcommand{\euY}{\ensuremath{\EuScript{Y}}}
\newcommand{\euX}{\ensuremath{\EuScript{X}}}

\newcommand{\euV}{\ensuremath{\EuScript{V}}}
\newcommand{\euU}{\ensuremath{\EuScript{U}}}
\newcommand{\euS}{\ensuremath{\EuScript{S}}}
\newcommand{\euT}{\ensuremath{\EuScript{T}}}
\newcommand{\euZ}{\ensuremath{\EuScript{Z}}}
\newcommand{\euW}{\ensuremath{\EuScript{W}}}

\newcommand{\eupsilon}{\ensuremath{\EuScript{E}}}

\newcommand{\CM}{\ensuremath{\textbf{Chow}_\infty}}

\def\hom{{\operatorname{Hom}}}

\def\ext{\operatorname{ext}}

\def\id{\operatorname{id}}

\def\op{\operatorname{op}}

\def\Mod{\operatorname{Mod}}

\def\dim{\operatorname{dim}}

\def\spec{\operatorname{Spec}}

\def\Fun{\operatorname{Fun}}

\def\map{\operatorname{map}}

\def\calg{\operatorname{CAlg}}

\def\h{\operatorname{\text{h}}}

\def\chow{\operatorname{CH}}

\def\LM{\operatorname{CHM}}
\def\Corr{\operatorname{Corr}}
\def\DM{\operatorname{DM}}
\def\SH{\operatorname{SH}}

\def\PrL{\operatorname{Pr^L}}
\def\PrR{\operatorname{Pr^R}}
\def\stb{\operatorname{stb}}
\def\gm{\operatorname{gm}}
\def\nislocst{\operatorname{Nis-locSt}}
\def\Ind{\operatorname{Ind}}

\def\Map{\operatorname{Map}}

\newcommand{\triplerightarrow}{%
\tikz[minimum height=0ex]
  \path[->]
   node (a)            {}
   node (b) at (1em,0) {}
  (a.north)  edge (b.north)
  (a.center) edge (b.center)
  (a.south)  edge (b.south);%
}

\makeatletter
  \def\subsection{\@startsection{subsection}{1}%
  \z@{.7\linespacing\@plus\linespacing}{.5\linespacing}%
  {\normalfont\bfseries\centering}}
\makeatother

\let\oldtocsection=\tocsection
\let\oldtocsubsection=\tocsubsection
\let\oldtocsubsubsection=\tocsubsubsection
\renewcommand{\tocsection}[2]{\hspace{0em}\oldtocsection{#1}{#2}}
\renewcommand{\tocsubsection}[2]{\hspace{1em}\oldtocsubsection{#1}{#2}}
\renewcommand{\tocsubsubsection}[2]{\hspace{2em}\oldtocsubsubsection{#1}{#2}}

\calclayout

\author[Dhyan Aranha, Chirantan Chowdhury]{Dhyan Aranha, Chirantan Chowdhury}
\email{dhyan.aranha@gmail.com,chirantanc474@gmail.com }

\begin{document}
\title[The Chow Weight structure for geometric motives of quotient stacks]{The Chow weight structure for geometric motives of quotient stacks }
\begin{abstract}
    We construct the Chow weight structure on the derived category of geometric motives $ \DM_{\gm} ([X/G], \Lambda)$ for $X$ a quasi-projective scheme over a field characteristic $0$, $G$ an affine algebraic group and $\Lambda$ an arbitrary commutative ring. In particular we also show that the heart of this weight structure recovers the category of Chow motives on $[X/G]$. 
\end{abstract}

\maketitle 
\tableofcontents

\section{Introduction}

The notion of weight structure on a triangulated category was introduced by Bondarko in \cite{bondarko_2010} and independently by Pauksztello \cite{Pauksztello} (under the name of "co-t-structures"). In \cite{hébert_2011} and \cite{Bondarkob} Chow weight structures for the derived category of Beilinson motives where constructed and in \cite{BondarkoIvanov} Chow weight structures for $\DM_{\text{cdh}} (- ,\Lambda)$ were constructed where $\Lambda$ is a general ring such that the characteristic of the base is invertible. 

The motivation for this note comes from the works of \cite{SVW}{Rem. II.4.15] and \cite{EberhardStroppel}[Rem. 4.8], where it is asked if there is a general way to put a Chow weight structure on derived category of (geometric) motives for quotient stacks. In this article we propose a definition for $\DM_{\gm} ([X/G] ,\Lambda)$, the derived category of geometric motives over a stack $[X/G]$ where $X$ is assumed to be quasi-projective (\cref{Definition:gemoetric motives}). Roughly speaking, $\DM_{\gm}([X/G], \Lambda)$ is the thick subcategory of $\DM([X/G], \Lambda)$ generated by (Tate twists of) motives of stacks which are smooth and quasi-projective over $[X/G]$. Our justification for this definition is that it is equivalent to the usual definition \cite{CiDeg}[Def. 2.3] when $G$ is trivial (\cref{{Lemma:quasi-projective geometric motives are the same over quasi-projective schemes}}). Our first main theorem is

\begin{introthm}{Theorem}
\label{Theorem:Introduction_existence_of_weight_str}
Suppose that $\euX = [X/G]$ where $X$ is a quasi-projective scheme over a field $k$ of characteristic $0$ and $G$ is an affine algebraic group acting on $X$.  Let $\Lambda$ be any commutative ring.  Then the $\infty$-category of geometric motives $\DM_{\gm} (\euX, \Lambda)$ admits a Chow weight structure $w_\text{Chow}$. 
\end{introthm}

The reason we call the weight structure constructed in Theorem \ref{Theorem:Introduction_existence_of_weight_str}, the \emph{Chow} weight structure is because of our second main theorem

\begin{introthm}{Theorem}
\label{Theorem:Introduction_weight_heart_is_chow_motives}
Suppose we are in the setup of Theorem \ref{Theorem:Introduction_existence_of_weight_str}. Then there is an equivalence
\begin{center}
    $ \h \DM_{\gm} (\euX, \Lambda)^{\heartsuit_{w}} \simeq \LM(\euX, \Lambda),$
\end{center}
where $\LM(\euX, \Lambda)$ denotes the category of classical Chow motives over $\euX$. (see Definition \ref{Definition:chow_motives}). 
\end{introthm}

Theorem \ref{Theorem:Introduction_existence_of_weight_str} appears as \cref{Theorem:Weight_structure_for_global_quotients} in the main text. The assumption that the field $k$ is of characteristic zero in Theorem \cref{Theorem:Introduction_existence_of_weight_str} is needed in two places in this work: Firstly, in order to prove the existence of proper $*$-pushforwards for $\DM_{\gm} (- , \Lambda)$ we rely on the results of \cite{AHLHR}. Secondly, because we need to use $G$-equivariant resolution of singularities. In particular to get this result in characteristic $p$, different arguments are needed and we think this is an interesting question. 

In fact in both \cite{SVW} and \cite{EberhardStroppel} weight structures were constructed on various subcategories of  $\DM([X/G], \Lambda)$ under various assumptions on the action of $G$ on $X$. We believe that with an appropriate version of Chow's lemma for stacks one can show that our category $\DM_{\gm} ([X/G], \Lambda)$ is equivalent to the category $\DM^{\text{Spr}}_G (X, \Lambda)$ of \cite{EberhardStroppel} when the base field $k$ is of characteristic $0$.

We now give a general outline of the article. In \cref{Section:DM for algebraic stacks} we introduce the $\infty$-category $\DM(\euX, \Lambda)$ and its six-functor formalism for a general class of stacks (so called Nis-loc stacks \cite{chowdhury2021motivic}). We also explain the equivalence 
\begin{center}
    $\DM(\euX, \Lambda) \simeq \Mod_{H\Lambda} (\SH(\euX))$
\end{center}
when $\euX$ is Nis-loc, which is presumably well known to experts, but we could not find a reference for this in the literature. 

In \cref{Section:Descent Results} we record various descent results which will be used to construct the Chow weight structure. Most important, will be the fact that $\DM (- ,\Lambda)$ on the category of Nis-loc stacks has $cdh$-descent (see \cref{Propostion:cdh_descent}). This will be used together with the existence of equivariant resolutions of singularities in characteristic $0$ to show that the $\infty$-category of Chow motives (see \cref{Definition:Infty category of chow motives}) generates the derived category of geometric motives for a quotient stack in a suitable sense.

In \cref{Section:Geometric motives} we introduce the category of geometric motives, $\DM_{\gm}(\euX, \Lambda)$, and consider how various operations in the six functor formalism on  $\DM(\euX , \Lambda)$ restrict to $\DM_{\gm} (\euX, \Lambda)$. One of the main results in this section is that $\DM_{\gm} (\euX, \Lambda)$ is stable under projective $*$-pushforwards. This result relies crucially on the results of \cite{AHLHR}.  The other important result in this section is that we show that the $\infty$-category of Chow motives $\CM(\euX, \Lambda)$ generates $\DM_{\gm} (\euX, \Lambda)$: \cref{Theorem:Generation_of_DM_gm_by_chow}. 

In \cref{Section:Mapping spectra and chow motives} we explain the connectivity of the mapping spectrum between any two objects of the $\infty$-category of Chow motives of a quotient stack. Along the way we will also explain the equivalence
\begin{center}
    $\pi_0 \map_{\DM(\euX, \Lambda)} (1_\euX (s)[2s+t], f^!1_B) \simeq \chow_s(\euX, t)_\Lambda.$
\end{center}
for a quotient stack $\euX = [X/G]$ which is well known in the case that $X$ is smooth \cite{RicharzScholbach}[Thm. 2.2.10] and \cite{KaRa}[12.4]. 

In \cref{Section:Weight Structure} we remind the reader about the definition of weight structures and prove the existence of the Chow weight structure for quotient stacks: \cref{Theorem:Weight_structure_for_global_quotients}. 

Finally in \cref{Section:Equivariant Motives} we explain the identification of the homotopy category of the weight heart of the Chow weight structure on $\DM_{\gm}([X/G], \Lambda)$ with the classically defined category of Chow motives. As expected in \cite{SVW}[II.4.15], when $\euX = BG$ for $G$ a linear algebraic group over a field $k$ and $ \Lambda = \mathbf{Q}$, Theorem \cref{Theorem:Introduction_weight_heart_is_chow_motives} gives an identification of the weight heart of our weight structure with Laterveer's category category of equivariant motives \cite{Laterveer} (See \cref{Corollary:Comparison with Laterveer's category}).

\subsection{Future work}

In forthcoming work with Alessandro D'Angelo, for a Nis-loc stack $\euX$ we will introduce an auxiliary category
\begin{equation*}
 \DM_{\gm, \ext} (\euX, \Lambda),
\end{equation*}
which is the Nis-loc extension (in the sense of \cite{chowdhury2021motivic}[Thm. 2.4.1]) of $\DM_{\gm} (-, \Lambda)$. This category inherits many nice properties such as $(-)_*$ functors for \emph{all} representable morphisms. We will use $\DM_{\gm, \ext} (\euX, \Lambda)$ to remove the quasi-projectivity hypothesis in \cref{Definition:gemoetric motives} and it will allow us to show that $\DM_{\gm} (\euX, \Lambda)$ is preserved under pushforwards for arbitrary representable morphisms. In particular, we will be able to use Bondarko's "gluing of weight structures" \cite{bondarko_2010}[8.2] to construct a canonical weight structure on $\DM_{\gm}(\euX, \Lambda)$ for for an arbitrary Nis-loc stack with affine stabilizers. 

Recently the notion of Perverse motive has been generalized to algebraic stacks in the work of Tubach \cite{tubach2023nori}. We hope that notion of weight structures for stacks will be useful in this context and it is something we are presently thinking about. 

Moreever, we also expect our work to have connections to the theory of \emph{Derived motivic measures}, as exposed in \cite{nanavaty2024weight}. That is to say we expect that there is a natural extension of the theory in loc. cit. to algebraic stacks and that the Chow weight structure we construct here should play a role. 

\subsection{Acknowledgements}
We would first like to take the opportunity to thank Alessandro D'Angelo for many conversations about the material in this note and for pointing out an important mistake in an earlier version of this work. We would like to thank Vova Sosnilo and Swann Tubach for very helpful email correspondences.  We  would also like to thank Marc Levine for patiently answering several asinine questions about motivic homotopy theory and Jochen Heinloth for many conversations about stacks\footnote{The authors were supported by the ERC through the project QUADAG.  This paper is part of a project that has received funding from the European Research Council (ERC) under the European Union's Horizon 2020 research and innovation programme (grant agreement No. 832833).  \\
\includegraphics[scale=0.08]{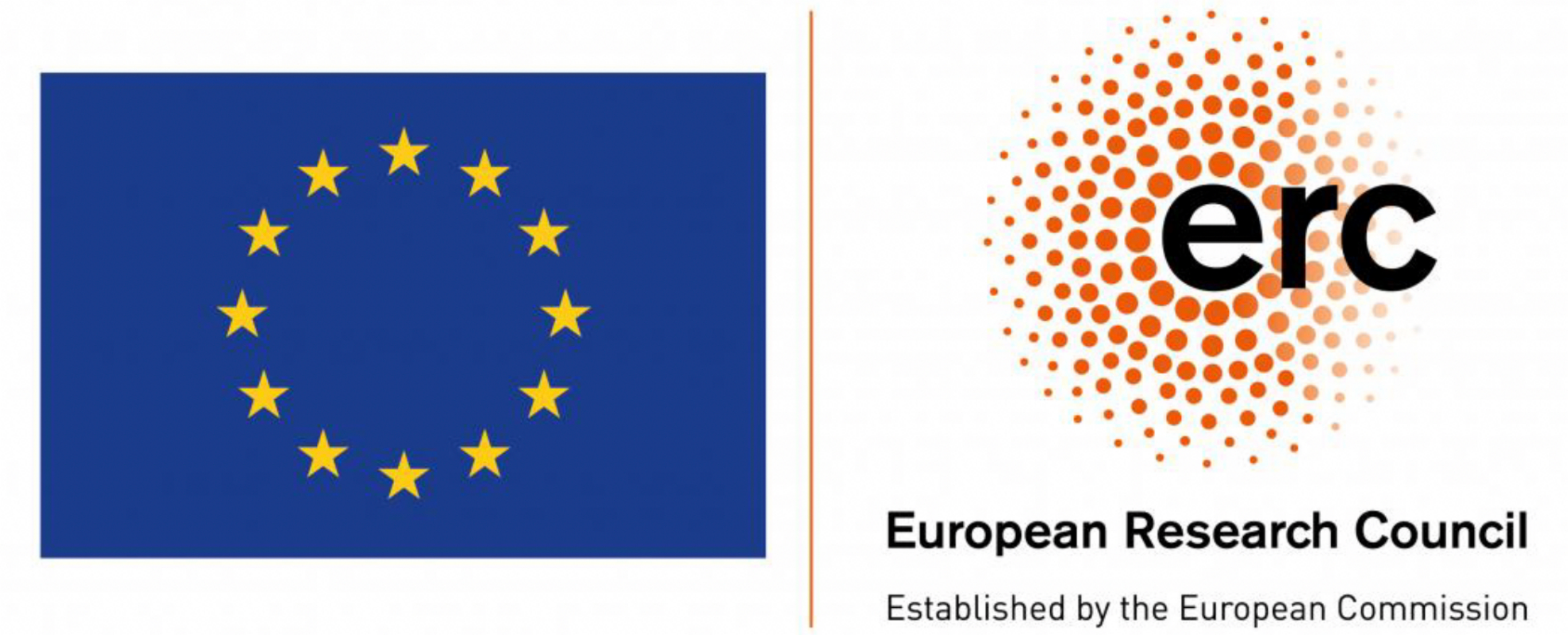}}.

\subsection{Notation} 

We will denote stacks by the letters $\EuScript{X}, \EuScript{Y}, \EuScript{Z}..$ etc. and denote schemes/algebraic spaces by the letters $X, Y Z..$ etc. All of our geometric objects will live over a base scheme $ B = \spec(k)$ where $k$ is an algebraically closed field. We will denote by $\Lambda$ a arbitrary commutative ring, such that if $\text{char}(k) = p > 0 $ we assume that $p$ is invertible in $\Lambda$. 

We will assume all of our stacks have affine diagonal and are of finite type over $B$. In particular by \cite{deshmukh2023motivic}[Thm. 1.2] all our stacks are Nis-loc (see \cref{Definition:Nis-loc_stack}). We will often still refer to the Nis-loc stack hypothesis anyway in many statements to reassure the reader.

\section{DM for algebraic stacks}
\label{Section:DM for algebraic stacks}

We will begin by recalling the construction of the category $\DM(\euX, \Lambda)$ and it's six functor formalism. Given a finite type $B$-scheme $S$, and a commutative ring $\Lambda$, it follows from the work of \cite{Spitzweck} and \cite{Hoyois} that there is a well defined motivic Eilenberg-MacLane spectrum $H\Lambda_S$. We make the following definition of the derived category of motives for finite type schemes over $B$.

\begin{Definition}
Let $X$ be a finite type scheme over $B$, and $\Lambda$ an arbitrary commutative ring. We define the derived category of motives with coefficients in $\Lambda$ to be
\begin{equation*}
\DM(X, \Lambda) := \Mod_{H\Lambda_X} (\SH(X)).
\end{equation*}
\end{Definition}
The category $\DM(X, \Lambda)$ is equivalent to the category $\DM_{\text{cdh}} (X, \Lambda)$ by \cite{integralmixedmotives}[Thm. 5.1]. In particular it has a six functor formalism which we will recall shortly. 

\begin{Remark}
\label{Remark:Connection_to_Beilinson_motives}
    In the case that $\Lambda = \mathbb{Q}$ it follows from \cite{integralmixedmotives} that $\DM (X, \Lambda)$ is equivalent to the category of Beilinson motives.
\end{Remark}

We now summarize the 6-functor formalism on schemes for $\DM(- , \Lambda)$ which follows from \cite{Spitzweck} and \cite{integralmixedmotives}. There is a functor
\begin{equation*}
    \DM^* : (\text{Sch}^\text{ft}_B) ^{\op} \rightarrow \calg(\PrL_{\stb, \Lambda})
\end{equation*}
from the category of finite type $B$-schemes to the $\infty$-category of presentable stable $\Lambda$-linear symmetric monoidal $\infty$-categories. Which sends $ X $ to the $\infty$-category $\DM(X, \Lambda)$, and $ f: X \rightarrow Y$ to $f^*$. This defines a motivic category in the sense of \cite{CiDeg} which has a six functor formalism:
\begin{enumerate}
\item For every $f: X\rightarrow Y$ the monoidal functor $f^*$ admits a right adjoint $f_*$.

\item For every smooth $f : X \rightarrow Y$ smooth, the functor $f^*$ admits a left adjoint $ f_\#$

\item For every separated morphism $f:X \rightarrow Y$ there is a functor $ f^! : \DM(Y, \Lambda) \rightarrow \DM(X, \Lambda)$ which admits a left adjoint $f_!$. 

\item There are adjoint functors $(\otimes, \underline{\map})$. 

\item For $f: X \rightarrow Y$, there exists a natural transformation
\begin{center}
   $ \alpha_f : f_! \rightarrow  f_*$
\end{center}
which is an equivalence when $\alpha$ is proper. 

\item (Purity) For any smooth separated morphism $ f: X \rightarrow Y$ of relative dimension $d$, there is a canonical natural isomorphism 
\begin{center}
$ \mathfrak{p}_f: f^* \overset{\simeq}{\rightarrow} f^! (-d)[-2d]$.
\end{center}

\item (Base Change) If 
  \begin{center}
	\begin{tikzpicture}
		\node (TL) at (0 , 1.2) {$X'$};
		\node (TR) at (1.2, 1.2) {$Y'$};
		\node (BL) at (0, 0){$X$};
        \node (BR) at (1.2,0){$Y$};
		
		\draw[->] (TL) -- (TR) node [midway, above] {$f'$};
		\draw[->] (TL) -- (BL)node [midway, left] {$g'$};
		\draw[->] (TR) -- (BR)node [midway, right] {$g$};
        \draw[->] (BL) -- (BR) node [midway, below] {$f$};
	\end{tikzpicture}
 \end{center}
is a cartesian square and $g$ is separated, then we have the following canonical equivalences:
\begin{eqnarray*}
     g'_! f'^* & \simeq &  f^* g_! \\
    f'_*g'^! & \simeq & g^! f_*. 
\end{eqnarray*}

\item (Projection formula) For any separated finite type morphism $ f: Y \rightarrow X$  we have the following equivalences 
\begin{eqnarray*}
    f_! (\mathcal{F} \otimes f^* (\mathcal{E})) & \simeq & f_! (\mathcal{F}) \otimes \mathcal{E} \\
    f^! \underline{\map}_{\DM(X, R)}  (\mathcal{E}, \mathcal{E}') & \simeq & \underline{\map}_{\DM(Y, R)}(f^*\mathcal{E}, f^!\mathcal{E}').
\end{eqnarray*}

\item (Localization) For $ i: Z \hookrightarrow X$ a closed immersion with open complement $ j: U \hookrightarrow X$ we have the following cofiber sequences
\begin{eqnarray*}
    i_! i^! \rightarrow  \id  \rightarrow   j_* j^* \\
    j_! j^! \rightarrow \id \rightarrow i_* i^*. 
\end{eqnarray*}

\item (Absolute purity) For any closed immersion $i: Z \hookrightarrow X$ between regular schemes of codimension $c$ there is an isomorphism
\begin{center}
$ i^! 1_X \simeq 1_Z (-c)[-2c]$
\end{center}
\end{enumerate}

\begin{Notation}
In order to avoid overloading notation we will simply write $\DM(-, \Lambda)$ for $\DM^*(-, \Lambda)$. 
\end{Notation}

Our goal now is to describe an extension of the functor $\DM(-, \Lambda)$ to a certain class of stacks introduced by Chowdhury \cite{chowdhury2021motivic} called Nis-loc stacks. We will first recall the definition.

\begin{Definition}
\label{Definition:Nis-loc_stack}
We say that an algebraic stack $\euX$ admits Nisnevich-local sections if there exists a morphism $ x : X \rightarrow \euX$ such that $X$ is a scheme and for any morphism $ y : Y \rightarrow \euX$ with $ Y$ a scheme, the induced map $ x': X \times_\euX Y \rightarrow Y$ admits Nisnevich-Local sections.  We say that an algebraic stack $\euX$ is \emph{Nis-loc} if there exists a smooth cover which admits Nisnevich-local sections. We will denote the $\infty$-category of Nis-loc stacks by $\nislocst$
\end{Definition}

\begin{Example}
The following example is from \cite{chowdhury2021motivic}[Cor. 2.3.6]. Let $X$ be a finite type scheme over $B$ and $G$ an affine algebriac group. Then $[X/G]$ is a Nis-loc stack.
\end{Example}

\begin{Example}
By \cite{deshmukh2023motivic}[Thm. 1.2]  any quasi-separated, finite type algebraic stack over $B$ with separated diagonal is Nis-loc.
\end{Example}

One can construct an extension of $\DM(-, \Lambda)$ to all locally finite type algebraic stacks over $B$ by considering the so called \textit{lisse-extension} as introduced in \cite{KaRa}[Constr. 12.1]. However, the \v{C}ech nerve of an arbitrary smooth cover will not be confinal in general and so we cannot compute this extension along arbitrary smooth covers. The reason for introducing the notion of Nis-loc stack is that they provide a convenient class of stacks where we can compute $\DM (- ,\Lambda)$  along \v{C}ech nerves of Nis-loc covers. 

\begin{Theorem}
\label{Theorem:Extension of DM to Nis-loc_stack}
The functor $\DM (-, \Lambda)$ extends to an $\infty$-sheaf
\begin{center}
   $ \DM_{\ext} (-, \Lambda): \nislocst^{\op} \rightarrow \calg(\PrL_{\stb, \Lambda}).$
\end{center}
Moreover, for any $ \euX \in \nislocst$ with a schematic Nis-loc atlas $\pi:  X \rightarrow \euX$ we can compute $\DM_{\ext}  (\euX, \Lambda)$  on the \v{C}ech nerve of $\pi$. That is 
\begin{center}
    $\DM_{\ext}  (\euX, \Lambda) \simeq \varprojlim \Big( \DM (X, \Lambda) \rightrightarrows \DM (X \times_\euX X , \Lambda) \triplerightarrow \cdots \Big) $. 
\end{center}
\end{Theorem}

\begin{proof}
The proof is the same as \cite{chowdhury2021motivic}[Cor. 2.5.1] with $\SH$ replaced by $ \DM (- , \Lambda)$. Note that as $\SH$ is a Nisnevich sheaf, by \cref{Proposition:Mod_commutes_with_limits} we see that $\DM(-,\Lambda)$ is a Nisnevich sheaf so we can apply \cite[Thm. 2.4.1]{chowdhury2021motivic} in mimicking the proof of \cite[Cor. 2.5.1]{chowdhury2021motivic}. We emphasize that the key result \cite{chowdhury2021motivic}[Thm. 2.4.1] is proved in the generality of an arbitrary $\infty$-sheaf. 
\end{proof}

\begin{Definition}
\label{Definition:DM_for_nis_loc_stacks}
Let $\euX$ be a Nis-Loc stack. We define the derived category of motives over $\euX$ with coefficients in an arbitrary commutative ring $\Lambda$  to be 
\begin{center}
    $\DM(\euX, \Lambda) : = \DM_{\ext} (\euX, \Lambda).$
\end{center}
\end{Definition}

\begin{Remark}
    We would like to take the opportunity to point out that the six functor formalism for $\DM$ of algebraic stacks has been considered in the literature in various places. In the case of quotient stacks there is \cite{SVW}[Ch. I]. For general stacks over $\mathbf{Q}$ there is \cite{RicharzScholbach}, and for general coefficients \cite{KaRa} [12.1]. Also there is the work \cite{hörmann2022derivator} in the setting of derivators.
\end{Remark}

The next couple of remarks record how the category defined in  \cref{Definition:DM_for_nis_loc_stacks} compares with other constructions in the literature. 

\begin{Remark}
As alluded too above, By \cite{KaRa}, one could just as well for an arbitrary locally finite type stack $\euX$  over $B$ define
\begin{center}
    $\DM_{\lhd}(\euX, \Lambda) : = \varprojlim_{(T,t)} \DM(T, \Lambda)$
\end{center}
where the limit is taken over the $\infty$-category $\text{Lis}_{\euX}$ of pairs $ (T, t)$ where $T$ is a scheme and $ t: T \rightarrow \euX$ is a smooth morphism. The same arguments used in \cite{chowdhury2021motivic}[Cor. 2.5.4] and \cite{KaRa}[Cor. 12.2.8] show that when $\euX$ is Nis-loc the categories $\DM_{\lhd}(\euX, \Lambda)$ and $\DM(\euX, \Lambda) $ are equivalent. 
\end{Remark}

\begin{Remark}
We can also define for an arbitrary locally finite type stack $\euX$ over $B$ the category
\begin{equation*}
\DM^!(\euX, \Lambda) : = \varprojlim_{{\text{Lis}}_{\euX}} \DM^!(S, \Lambda)
\end{equation*}
in $\PrR_{\stb, \Lambda}$. When $\euX$ is Nis-loc, the purity isomorphism implies that 
\begin{center}
$\DM(\euX, \Lambda) \simeq \DM^! (\euX, \Lambda)$. 
\end{center}
In particular when $\Lambda = \mathbf{Q}$ and $\euX$ is Nis-loc then  \cref{Definition:DM_for_nis_loc_stacks} agrees with the the derived category of motives constructed in \cite{RicharzScholbach}. 
\end{Remark}

Next we would like to give a more global description of $\DM(\euX, \Lambda)$. That is we would like to describe  $\DM(\euX, \Lambda)$ as a category of modules in $\SH(\euX)$ over some motivic $\mathbf{E}_\infty$-ring spectrum. To do this we will first start with a purely categorical statement. 

\begin{Proposition}
\label{Proposition:Mod_commutes_with_limits}
Suppose that $\mathcal{C}^{\otimes} \in  \calg (\text{Cat}^\otimes_\infty)$ is the limit of a diagram $q : I \rightarrow \calg (\text{Cat}^\otimes_\infty)$. Let $\Mod (\mathcal{C})$ be as in \cite{higher_algebra}[Def. 3.3.3.8], then we have a canonical equivalence
\begin{center}
    $\Mod (\mathcal{C}) \simeq \varprojlim_{i \in I} \Mod(\mathcal{C}_i)$
\end{center}
where $q(i) := \mathcal{C}^{\otimes}_i$.
\end{Proposition}
\begin{proof}
The $\infty$-category of modules associated to a symmetric monoidal $\infty$-category $\mathcal{C}$ is equivalent to the $\infty$-category $\operatorname{Alg}_{\mathbf{Pf}^{\otimes}}(\mathcal{C})$ of algebra objects associated to the $\infty$-operad $\mathbf{Pf}^{\otimes}$ (\cite[Section 9.4.1.2]{Robalothesis}). Thus we can realize $\Mod(\mathcal{C})$ as a full  subcategory of the functor category $\operatorname{Fun}(\mathbf{Pf}^{\otimes},\mathcal{C}^{\otimes})$ spanned by objects $ p : \mathbf{Pf}^{\otimes} \to \mathcal{C}^{\otimes}$ which commute with the usual projection maps to $N(\operatorname{Fin}_*)$. Firstly, we see that we have the following chain of equivalences :
\begin{equation}
    \Fun(\mathbf{Pf}^{\otimes},\mathcal{C}^{\otimes}) \simeq \Fun(\mathbf{Pf}^{\otimes},\varprojlim_{i \in I} \mathcal{C}_i^{\otimes}) \simeq \varprojlim_{i \in I}\Fun(\mathbf{Pf}^{\otimes},\mathcal{C}_i^{\otimes}).
\end{equation}
In order to get the equivalence on the level of module categories, we are reduced to check if $\{p_i: \mathbf{Pf}^{\otimes} \to \mathcal{C}_i^{\otimes}\}_{i \in I}$ is a compatible family of morphisms commuting to $N(\operatorname{Fin}_*)$, then the limit morphims $p : \mathbf{Pf}^{\otimes} \to \mathcal{C}^{\otimes}$ commutes with projection to $N(\operatorname{Fin}_*)$. This is because $\calg(\text{Cat}^{\otimes}_{\infty})$ admits limits (\cite[Proposition 3.2.2.1]{higher_algebra}).
\end{proof}

\begin{Lemma}
\label{Lemma:Extension_of_motivic_cohomology_to_stacks}
Let $\euX$ be a Nis-loc stack and $\Lambda$ an arbitrary commutative ring. Then there is a canonically defined object $H\Lambda_\euX \in \calg(\SH(\euX)) $ whose restriction along any morphism $ f: U \rightarrow \euX$ from a scheme $U$ is canonically equivalent to $H\Lambda_U \in \calg(\SH(U))$. 
\end{Lemma}

\begin{proof}
Consider the ring spectrum $H\Lambda_B \in \calg(\SH(B))$ constructed in \cite{Spitzweck} and \cite{Hoyois}. We define $H\Lambda_\euX$ to be $ f^* H\Lambda_B$. Where $f:\euX \rightarrow B$ is the structure morphism. It follows directly from the definition of $\SH(\euX)$ \cite{chowdhury2021motivic}[Cor. 2.5.1] that $f^*$ is a symmetric monoidal functor and thus $H\Lambda_\euX $ is contained in $\calg(\SH(\euX))$. Then via \cite{Hoyois}[Lem. 20] it follows that $H\Lambda$ has the desired property. 
\end{proof}

\begin{Theorem}
\label{Theorem:Mod_agrees_with_Kan_extension}
Let $\euX$ be a Nis-loc stack over $B$ and $\Lambda$ an arbitrary commutative ring. Then we have the following canonical equivalence
\begin{center}
    $\Mod_{H\Lambda_{\euX}} (\SH(\euX)) \simeq \DM(\euX, \Lambda)$.
\end{center}
\end{Theorem}

\begin{proof}

As $\euX$ is Nis-loc, for an atlas $x: X \to \euX$ admitting Nisnevich-local sections, we have the following equivalence:
\begin{equation*}
    \SH(\euX) \simeq  \varprojlim \Big( \SH (X) \rightrightarrows \SH(X \times_\euX X) \triplerightarrow \cdots \Big).
\end{equation*}
Applying \cref{Proposition:Mod_commutes_with_limits} to $\mathcal{C}= \SH(\euX),\  I = N(\Delta),\ \mathcal{C}_i= \SH(X^i_{\euX})$, where $X^i_{\euX}:= X \times_{\euX} X \cdots_{\text{(i+1)-times}} X$, we get the equivalence, 
\begin{equation*}
    \Mod(\SH(\euX)) \simeq \varprojlim_{i\in \Delta} \Mod(\SH(X^i_{\euX})).
\end{equation*}
Taking the fiber of the equivalence over the canonical Eilenberg-Maclane spectrum $H\Lambda_{\euX} \simeq \varprojlim_{i \in \Delta} H\Lambda_{X^i_{\euX}}$ (\cref{Lemma:Extension_of_motivic_cohomology_to_stacks}), we get that
\begin{equation*}
    \Mod_{H\Lambda_{\euX}}(\SH(\euX)) \cong \varprojlim_{i \in I} \Mod_{H(X^i_{\euX})}(\SH(X^i_{\euX}).
\end{equation*}
By definition of $\DM$ on the level of schemes along with \cref{Theorem:Extension of DM to Nis-loc_stack}, we get that
\begin{equation*}
    \Mod_{H\Lambda_{\euX}}(\SH(\euX)) \simeq \DM(\euX,\Lambda)
\end{equation*}
completing the proof.
\end{proof}

\begin{Remark} One could envision another construction of $\DM(\euX, \Lambda)$ more along the lines of \cite{CiDeg}[Def. 11.1.1]. That is consider the category of stable motivic complexes on a given Nis-loc stack. It would be interesting to compare this with \cref{Definition:DM_for_nis_loc_stacks}. 
\end{Remark}

We will now explain how the six functor formalism for $\DM(-, \Lambda)$ on schemes generalizes to $\nislocst$. 

\begin{Proposition}
\label{Proposition:4_functors_for_DM}
(4-functors)
The functor
\begin{equation*}
    \DM(- ,\Lambda) : \nislocst^{\op} \rightarrow \calg(\PrL_{\stb, \Lambda})
\end{equation*}
has the following 4-functor formalism:
\begin{enumerate}
    \item For every morphism $ f: \euX \rightarrow \euY$ in $\nislocst$ we have a pair of adjoints $(f^*, f_*)$, such that $f^*$ is symmetric monoidal. 

    \item For every $\euX$ in $\nislocst$ there are functors 
    \begin{eqnarray*}
    - \otimes - : \DM(\euX, \Lambda) \times \DM(\euX, \Lambda) \rightarrow \DM(\euX, \Lambda)  \\
    \underline{\map}(-, -) : \DM(\euX, \Lambda) ^{\op}  \times \DM(\euX, \Lambda) \rightarrow \DM(\euX, \Lambda)
    \end{eqnarray*}
    which form an adjoint pair $(\otimes, \underline{\map})$. i.e. $\DM(\euX, \Lambda)$ is a closed symmetric monoidal $\infty$-category. 

\end{enumerate}
\end{Proposition}

\begin{proof}
The existence of $f^*$ and the fact that it is symmetric monoidal follows directly from \cref{Definition:DM_for_nis_loc_stacks}. Moreoever since $f^*$ is colimit preserving by Lurie's adjoint functor theorem \cite{higher_topos_theory}[Cor. 5.5.2.9] there exists a right adjoint which we denote by $f_*$. 

To see that $\DM(\euX, \Lambda)$ is a closed symmetric monoidal $\infty$-category one can simply use the proof of \cite{chowdhury2021motivic}[Prop. 2.5.6] replacing $\SH$ with $\DM$. 
\end{proof}

The next proposition records the existence of $f_{\#}$ for smooth and representable morphisms and its properties under base change. This will be important for defining the category of \emph{geometric motives}. 

\begin{Proposition}
    \label{Proposition:Existence_of_motives}
Let $ f: \euX \rightarrow \euY$ be a smooth representable morphism in $\nislocst$. Then $f^*$ admits a left adjoint $f_{\#}$. Moreover for any cartesian square 
\begin{center}
	\begin{tikzpicture}
		\node (TL) at (0 , 1.2) {$\euX'$};
		\node (TR) at (1.2, 1.2) {$\euY'$};
		\node (BL) at (0, 0){$\euX$};
        \node (BR) at (1.2,0){$\euY$};
		
		\draw[->] (TL) -- (TR) node [midway, above] {$f'$};
		\draw[->] (TL) -- (BL)node [midway, left] {$g'$};
		\draw[->] (TR) -- (BR)node [midway, right] {$g$};
        \draw[->] (BL) -- (BR) node [midway, below] {$f$};
	\end{tikzpicture}
 \end{center}
 in $\nislocst$, where $f$ is smooth and representable and $g$ is arbitrary, we have an equivalence
 \begin{equation*}
     g^*f_{\#} \simeq f'_{\#} g'^*. 
 \end{equation*}
\end{Proposition}

\begin{proof}
The proof of \cite{chowdhury2021motivic}[Prop. 4.1.2] goes through in this situation with $\SH$ replaced with $\DM$. 
\end{proof}

\begin{Construction}
(Tate twists)
Let $\euX$ be a Nis-loc stack and  $p: {\mathbf{G}_m} \times \euX \rightarrow \euX $ be the projection. Let $ \mathcal{F} \in \DM(\euX, \Lambda)$, then the map induced by the counit 
\begin{center}
    $p_{\#} p^* \mathcal{F}[-1] \rightarrow \mathcal{F}[-1] $
\end{center}
is a split monomorphism. We denote the complementary summand by $\mathcal{F} (1) $. It follows by an easy descent argument that the functor
\begin{center}
$\mathcal{F} \mapsto \mathcal{F}(1)$
\end{center}
is invertible. In particular we have Tate twists for all $ n \in \mathbf{Z}$. 
\end{Construction}

We now record the existence of exceptional functors, base change, projection formulas and proper pushforward formulas. The formulas stated here are in the full generality of not necessarily representable morphism. The proofs in this case will be contained in forthcoming work of the second named author and Alessandro D'Angelo. We wish to emphasize that for the present work \emph{we do not need this level of generality} and for us it enough to assume that all morphisms are \emph{representable} and in this case the arguments of \cite{chowdhury2021motivic} apply.  

\begin{Proposition}
\label{Proposition:Existence_of_exceptional_functors}
For any locally of finite type morphism $ f: \euX \rightarrow \euY$ of Nis-loc stacks there exist functors
\begin{eqnarray*}
    f_! : \DM(\euX, \Lambda) \rightarrow \DM(\euY, \Lambda) \\
    f^! : \DM(\euY, \Lambda) \rightarrow \DM(\euX, \Lambda)
\end{eqnarray*}
which form an adjoint pair $ (f_! , f^!)$ and satisfy:
\begin{enumerate}
    \item (Projection formula)  Let $\mathcal{E}, \mathcal{E}' \in \DM(\euY, \Lambda) $ and $\mathcal{F} \in \DM(\euX, \Lambda)$ we have the following equivalences
    \begin{eqnarray*}
        f_! (\mathcal{F} \otimes f^* (\mathcal{E}) & \simeq & f_!(\mathcal{F})\otimes \mathcal{E} \\
        f^! \underline{\map}_{\DM(\euY, \Lambda)} (\mathcal{E}, \mathcal{E}') & \simeq & \underline{\map}_{\DM(\euX, \Lambda)} (f^* \mathcal{E}, f^! \mathcal{E}'). 
    \end{eqnarray*}
    \item (Base change) If 
    \begin{center}
	\begin{tikzpicture}
		\node (TL) at (0 , 1.2) {$\euX'$};
		\node (TR) at (1.2, 1.2) {$\euY'$};
		\node (BL) at (0, 0){$\euX$};
        \node (BR) at (1.2,0){$\euY$};
		
		\draw[->] (TL) -- (TR) node [midway, above] {$f'$};
		\draw[->] (TL) -- (BL)node [midway, left] {$g'$};
		\draw[->] (TR) -- (BR)node [midway, right] {$g$};
        \draw[->] (BL) -- (BR) node [midway, below] {$f$};
	\end{tikzpicture}
 \end{center}
 is a cartesian square in $\nislocst$, of locally finite type morphisms. We have the following equivalences 
 \begin{eqnarray*}
     g'_! f'^*  \simeq  f^* g_!\\
     f'_* g'^! \simeq g^! f_*
 \end{eqnarray*}

 \item (Proper pushforward) If $f: \euX \rightarrow \euY$ is proper and representable, then there exists  a natural isomorphism
 \begin{center}
    $ \alpha_f : f_! \simeq f_*. $
 \end{center}
\end{enumerate}
\end{Proposition}

\begin{proof}
The proof of this in full generality will be contained in forthcoming work \cite{ChowdhuryD'Angelo}.

For the purposes of this article, we will only need to use the result in the representable case and for this the same argument used in \cite{chowdhury2021motivic}[Thm. 4.5.1] carry over verbatim.
\end{proof}

\begin{Proposition}
\label{Proposition:purity}
(Purity) The Nisnevich sheaf $\DM_{\gm} (- , \Lambda)$ on $\nislocst$ is oriented. Moreover we have 

\begin{enumerate}
    \item For any smooth representable morphism of relative dimension $d$, there is a natural isomorphism
\begin{center}
    $f^! \simeq f^* (d)[2d]$. 
\end{center}
    \item For a closed immersion $ i: \euZ \hookrightarrow \euX$ between regular Nis-loc stacks of codimension $c$ we have
\begin{center}
    $i^! 1_\euX \simeq 1_\euZ (-c)[-2d]$. 
\end{center}
\end{enumerate}
\end{Proposition}

\begin{proof}
For the claim about the orientation we refer to reader to \cite{aranha2022localization}[Rem. 1.7]. The proof of \cite{chowdhury2021motivic}[4.4.1] works in this case with $\SH$ replaced by $\DM$. We note that the separated hypothesis is not needed in loc. cit. because on the level of $B$-schemes of finite type we already have the existence of the exceptional functors. 

 We note that (2) is a direct consequence of (1). 
\end{proof}

\begin{Proposition}
    \label{Proposition:Localization}
(Localization)
Let $\euX$ be a Nis-loc stack. Suppose $j: \euU \hookrightarrow \euX$ is an open immersion with closed complement $ i : \euZ \hookrightarrow \euX$ then we have the following cofibers
\begin{eqnarray*}
     i_! i^! \rightarrow  \id  \rightarrow   j_* j^* \\
    j_! j^! \rightarrow \id \rightarrow i_* i^*. 
\end{eqnarray*}
\end{Proposition}

\begin{proof}
The same proof as \cite{chowdhury2021motivic}[Prop. 4.2.1] works with $\SH$ replaced by $\DM(-, \Lambda)$. We note that in this case the inclusions $i$ and $j$ are representable thus the proof of \cite{chowdhury2021motivic}[Thm. 3.1.1] applied to $\DM(-, \Lambda)$ will also construct the exceptional functors. 
\end{proof}

\begin{Proposition}
    \label{Proposition:homotopy_invariance}
(Homotopy invariance)
For any Nis-loc stack $\euX$, the projection $\pi: \mathbf{A}^1_\euX \rightarrow \euX$ induces a fully-faithful functor $\pi^* : \DM(\euX, \Lambda) \rightarrow \DM(\mathbf{A}^1_\euX, \Lambda)$.
\end{Proposition}

\begin{proof}
The same proof of \cite{chowdhury2021motivic}[Prop. 4.2.2] works with $\SH$ replaced by $\DM$. 
\end{proof}

\begin{Corollary}
\label{Corollary:Recollement for DM}
(Recollement) Suppose we have a diagram in $\nislocst$
    \begin{center}
        $ \euU \overset{j}{\hookrightarrow} \euX \overset{i}{\hookleftarrow} \euZ:= \euX - \euU$
    \end{center}
where $j$ is an open immersion and $i$ is a closed immersion. Then for $\DM (-, \Lambda)$ the following conditions are satisfied:
\begin{enumerate}
    \item The functor $i_* \simeq i_!$ admits a left adjoint $i^*$  and a right adjoint $i^!$.
    \item The functor $j^* \simeq j^!$ admits a right adjoint $j_*$ and a left adjoint $j_!$. 
    \item There is an equivalence $j^*i_* \simeq 0$. 
    \item We have the following localization triangles
    \begin{eqnarray*}
     i_! i^! \rightarrow  \id  \rightarrow   j_* j^* \\
    j_! j^! \rightarrow \id \rightarrow i_* i^*. 
    \end{eqnarray*}
    \item The functors $i_*, j_*$ and $j_!$ are all full embeddings.
\end{enumerate}
\end{Corollary}

\begin{proof}
 Both $(1)$ and $(2)$ follow \cref{Proposition:Existence_of_exceptional_functors}, \cref{Proposition:4_functors_for_DM} and \cref{Proposition:purity}. Claims $(3)$ and $(5)$ follow from the base change equivalence in \cref{Proposition:Existence_of_exceptional_functors}. Finally $(4)$ is \cref{Proposition:Localization}. 
\end{proof}

\section{Descent results}
\label{Section:Descent Results}

In this section we discuss $cdh$-descent and Nisnevich descent for $\DM(-, \Lambda)$. Khan in \cite{Khansix} has shown that $\DM(-, \Lambda)$ when restricted to algebraic spaces satisfies $cdh$-descent. What we say in this section follows easily from Khan's  work \cite{Khansix} but we will review the arguments here for completeness. For our purposes it will not be necessary to consider $cdh$-squares and Nisnevich squares for arbitrary morphisms of stacks, we will only need to consider representable $cdh$ and Nisnevich squares.

Recall from \cite{HoyoisSix} that the \emph{constructible topology} on $\text{AlgSp}_B$  is the coarsest topology such that 
\begin{enumerate}
    \item The empty sieve covers the empty algebraic space.
    \item If $Z \hookrightarrow X$ is a closed immersion with open complement $ U \hookrightarrow X$, $\{ U \hookrightarrow X, Z \hookrightarrow X\} $ generates a covering sieve. 
\end{enumerate}

\begin{Lemma}
\label{Lemma:Constructible_topology_is_conservative}
Let $\{ f_i : U_i \rightarrow S\} $  be a constructible cover of $S$ in $\text{AlgSp}_B$. Then the family of functors $\{ f_i^* :\DM(S, \Lambda) \rightarrow \DM(U_i, \Lambda) \}$ is conservative 
\end{Lemma}

\begin{proof}
This follows directly from \cref{Proposition:Localization}. 
\end{proof}

The following is \cite{Khansix}[Thm. 2.51]

\begin{Proposition}
\label{Proposition:cdh_descent_for_schemes}
Suppose that cartesian square
 \begin{eqnarray*}
	\begin{tikzpicture}
		\node (TL) at (0 , 1.2) {$T$};
		\node (TR) at (1.2, 1.2) {$Y$};
		\node (BL) at (0, 0){$Z$};
        \node (BR) at (1.2,0){$X$};
		
		\draw[->] (TL) -- (TR) node [midway, above] {$k$};
		\draw[->] (TL) -- (BL)node [midway, left] {$g$};
		\draw[->] (TR) -- (BR)node [midway, right] {$f$};
        \draw[->] (BL) -- (BR) node [midway, below] {$i$};
	\end{tikzpicture}
 &  & 
 \begin{tikzpicture}
		\node (TL) at (0 , 1.2) {$T$};
		\node (TR) at (1.2, 1.2) {$Y$};
		\node (BL) at (0, 0){$U$};
        \node (BR) at (1.2,0){$X$};
		
		\draw[->] (TL) -- (TR) node [midway, above] {$k$};
		\draw[->] (TL) -- (BL)node [midway, left] {$(\text{resp.      } g$};
		\draw[->] (TR) -- (BR)node [midway, right] {$f\ \ )$};
        \draw[->] (BL) -- (BR) node [midway, below] {$j$};
	\end{tikzpicture}
\end{eqnarray*}
in $\text{AlgSp}_B$ is a $cdh$-square (resp. Nisnevich square). Then we have a canonical equivalence 
\begin{center}
    $ \DM(X, \Lambda) \simeq \DM(Z, \Lambda) \times_{\DM(T, \Lambda) }\DM(Z, \Lambda)$
\end{center}
(resp. 
\begin{center}
 $ \DM(X, \Lambda) \simeq \DM(U, \Lambda) \times_{\DM(T, \Lambda) }\DM(Y, \Lambda)$)
\end{center}
of $\infty$-categories.
\end{Proposition}

\begin{proof}
The proof is the same as the proof \cite{HoyoisSix}[Prop. 6.24]. We prove the $cdh$-statement the Nisnevich result is analogous. That is, by \cite{DAGVII} it is enough to show 
\begin{enumerate}
    \item  The pair $(i^*, f^*)$ is conservative.
    \item Given $ \mathcal{F}_Z \in \DM(Z, \Lambda), \mathcal{F}_Y \in \DM(Y, \Lambda), \mathcal{F}_T \in \DM(T, \Lambda)$ and $ g^* (\mathcal{F}_Z  )\simeq \mathcal{F}_T \simeq k^*(\mathcal{F}_Y)$, if 
    \begin{center}
        $\mathcal{F}_X = i_* \mathcal{F}_Z \times_{(fk)_* \mathcal{F}_T} f_* \mathcal{F}_Y,$
    \end{center}
    then the maps 
    \begin{center}
        $i^* \mathcal{F}_X \rightarrow \mathcal{F}_Z$ and $f^* \mathcal{F}_X \rightarrow \mathcal{F}_Y$
    \end{center}
    induced by the canonical projections are equivalences. 
\end{enumerate}

Part $(1)$ follows directly from \cref{Lemma:Constructible_topology_is_conservative}. Part $(2)$ follows from first noting that proper base change \cref{Proposition:Existence_of_exceptional_functors} and the fact that $i_*$ is fully faithful implies that $ i^*\mathcal{F}_X \rightarrow \mathcal{F}_Z$ is an equivalence. To see the later equivalence we note that it follows via smooth base change \cref{Proposition:Existence_of_motives}. 
\end{proof}

The next definitions are natural generalizations of the notions of $cdh$-square and Nisnevich square to $\nislocst$. 

\begin{Definition}
\label{Definition:cdh_square}
    A cartesian diagram of algebraic stacks in $\nislocst$  
    \begin{center}
	\begin{tikzpicture}
		\node (TL) at (0 , 1.2) {$\euT$};
		\node (TR) at (1.2, 1.2) {$\euY$};
		\node (BL) at (0, 0){$\euZ$};
        \node (BR) at (1.2,0){$\euX$};
		
		\draw[->] (TL) -- (TR) node [midway, above] {$k$};
		\draw[->] (TL) -- (BL)node [midway, left] {$g$};
		\draw[->] (TR) -- (BR)node [midway, right] {$f$};
        \draw[->] (BL) -- (BR) node [midway, below] {$i$};
	\end{tikzpicture}
\end{center}
is called a representable $cdh$-square if:
\begin{enumerate}
\item The morphism $f$ is representable proper and surjective .
\item The morphism $i$ is a closed immersion.
\item The restriction of $f$ to $\euX-\euZ$ is an isomorphism. 
\end{enumerate}
\end{Definition}

\begin{Definition}
\label{Definition:Nis_sqaure}
We say that a cartesian diagram of algebraic stacks in $\nislocst$ 
\begin{center}
	\begin{tikzpicture}
		\node (TL) at (0 , 1.2) {$\euT$};
		\node (TR) at (1.2, 1.2) {$\euY$};
		\node (BL) at (0, 0){$\euU$};
        \node (BR) at (1.2,0){$\euX$};
		
		\draw[->] (TL) -- (TR) node [midway, above] {$k$};
		\draw[->] (TL) -- (BL)node [midway, left] {$g$};
		\draw[->] (TR) -- (BR)node [midway, right] {$f$};
        \draw[->] (BL) -- (BR) node [midway, below] {$j$};
	\end{tikzpicture}
\end{center}
is called a representable Nisnevich square if:
\begin{enumerate}
\item The morphism $f$ is representable \'etale morphism.
\item The morphism $j$ is an open immersion.
\item The restriction of $f$ to $(\euX-\euU)_{\text{red}}$ is an isomorphism. 
\end{enumerate}
\end{Definition}

\begin{Proposition}
\label{Propostion:cdh_descent}
Given a $cdh$-square (resp. Nisnevich square)
\begin{eqnarray*}
	\begin{tikzpicture}
		\node (TL) at (0 , 1.2) {$\euT$};
		\node (TR) at (1.2, 1.2) {$\euY$};
		\node (BL) at (0, 0){$\euZ$};
        \node (BR) at (1.2,0){$\euX$};
		
		\draw[->] (TL) -- (TR) node [midway, above] {$k$};
		\draw[->] (TL) -- (BL)node [midway, left] {$g$};
		\draw[->] (TR) -- (BR)node [midway, right] {$f$};
        \draw[->] (BL) -- (BR) node [midway, below] {$i$};
	\end{tikzpicture}
 &  & 
 \begin{tikzpicture}
		\node (TL) at (0 , 1.2) {$\euT$};
		\node (TR) at (1.2, 1.2) {$\euY$};
		\node (BL) at (0, 0){$\euU$};
        \node (BR) at (1.2,0){$\euX$};
		
		\draw[->] (TL) -- (TR) node [midway, above] {$k$};
		\draw[->] (TL) -- (BL)node [midway, left] {$(\text{resp.      } g$};
		\draw[->] (TR) -- (BR)node [midway, right] {$f\ \ )$};
        \draw[->] (BL) -- (BR) node [midway, below] {$j$};
	\end{tikzpicture}
\end{eqnarray*}
in $\nislocst$. There is a canonical equivalence
\begin{center}
    $ \DM(\euX, \Lambda) \simeq \DM(\euZ, \Lambda) \times_{\DM(\euT, \Lambda) }\DM(\euY, \Lambda)$
\end{center}
(resp. 
\begin{center}
 $ \DM(\euX, \Lambda) \simeq \DM(\euU, \Lambda) \times_{\DM(\euT, \Lambda) }\DM(\euY, \Lambda)$)
\end{center}
of $\infty$-categories.
\end{Proposition}

\begin{proof}
We will only prove the $cdh$-case, the Nisnevich case is entirely analogous. Let $ \pi: X \rightarrow \euX$ be a Nis-loc atlas of $\euX$. Then by \cref{Theorem:Extension of DM to Nis-loc_stack} there is an equivalence
\begin{center}
    $\DM(\euX, \Lambda) \simeq \varprojlim_{n \in \Delta} \DM(X_n, \Lambda).$
\end{center}
But now for each $ n \in \Delta$ the induced square
\begin{center}
	\begin{tikzpicture}
		\node (TL) at (0 , 1.2) {$T_n$};
		\node (TR) at (1.2, 1.2) {$Y_n$};
		\node (BL) at (0, 0){$Z_n$};
        \node (BR) at (1.2,0){$X_n$};
		
		\draw[->] (TL) -- (TR) node [midway, above] {$k_n$};
		\draw[->] (TL) -- (BL)node [midway, left] {$g_n$};
		\draw[->] (TR) -- (BR)node [midway, right] {$f_n$};
        \draw[->] (BL) -- (BR) node [midway, below] {$i_n$};
	\end{tikzpicture}
\end{center}
is a $cdh$-square of algebraic spaces (resp. Nisnevich square). Thus by \cref{Proposition:cdh_descent_for_schemes} there is a canonical equivalence
\begin{center}
   $ \DM(X_n, \Lambda) \simeq \DM(Z_n, \Lambda) \times_{\DM(T_n, \Lambda)} \DM(Y_n, \Lambda). $
\end{center}
We can then rewrite $\DM(\euX, \Lambda)$ as
\begin{center}
   $ \DM(\euX, \Lambda) \simeq \varprojlim_{n \in \Delta} \DM(Z_n, \Lambda) \times_{\DM(T_n, \Lambda)} \DM(Y_n, \Lambda)$. 
\end{center}
But now since limits commute with fiber products we are done.
\end{proof}

Given $cdh$-square as in \cref{Definition:cdh_square}, and setting $a \simeq ig \simeq fk$. For any $ \mathcal{M} \in \DM(\euX, \Lambda)$ we can form a commutative square

\begin{equation}
\label{Equation:cdh_square}
\begin{tikzpicture}
		\node (TL) at (0 , 1.2) {$\mathcal{M}$};
		\node (TR) at (2, 1.2) {$f_* f^* (\mathcal{M})$};
		\node (BL) at (0, 0){$i_* i^*(\mathcal{M})$};
        \node (BR) at (2,0){$a_* a^*( \mathcal{M})$};
		
		\draw[->] (TL) -- (TR);
		\draw[->] (TL) -- (BL);
		\draw[->] (TR) -- (BR);
        \draw[->] (BL) -- (BR);
	\end{tikzpicture}
\end{equation}
by considering the various unit and counit natural transformation associated to the adjunctions involved. 

\begin{Corollary}
The square in \cref{Equation:cdh_square} is cartesian. 
\end{Corollary}

We will also be interested in the case of a Nisnevich square as in \cref{Definition:Nis_sqaure}. That is, again setting $ a \simeq jg \simeq fk$, for any $ \mathcal{M} \in \DM(\euX, \Lambda)$ we can form the commutative square 

\begin{equation}
\label{Equation:nis_square}
\begin{tikzpicture}
		\node (TL) at (0 , 1.2) {$\mathcal{M}$};
		\node (TR) at (2, 1.2) {$f_* f^* (\mathcal{M})$};
		\node (BL) at (0, 0){$j_* j^*(\mathcal{M})$};
        \node (BR) at (2,0){$a_* a^*( \mathcal{M}).$};
		
		\draw[->] (TL) -- (TR);
		\draw[->] (TL) -- (BL);
		\draw[->] (TR) -- (BR);
        \draw[->] (BL) -- (BR);
	\end{tikzpicture}
\end{equation}
Applying the functor $\underline{\map}_{\DM(\euX, \Lambda)} (- , 1)$ we also get a square
\begin{equation}
\label{Equation:exceptional_nis_square}
\begin{tikzpicture}
		\node (TL) at (0 , 1.2) {$ a_\# a^*( \mathcal{M})$};
		\node (TR) at (2, 1.2) {$f_\# f^* (\mathcal{M})$};
		\node (BL) at (0, 0){$j_\# j^*(\mathcal{M})$};
        \node (BR) at (2,0){$\mathcal{M}.$};
		
		\draw[->] (TL) -- (TR);
		\draw[->] (TL) -- (BL);
		\draw[->] (TR) -- (BR);
        \draw[->] (BL) -- (BR);
	\end{tikzpicture}
\end{equation}

\begin{Corollary}
\label{Corollary:Mayer-Vietoris_squares}
The square in \cref{Equation:nis_square} is cartesian and hence so is \cref{Equation:exceptional_nis_square}.  
\end{Corollary}

\cref{Corollary:Mayer-Vietoris_squares} can be used to to show that for geometric motives we get a Mayer-Vietoris cofiber sequence.

\begin{Corollary}
\label{Corollary:Mayer-Vietoris_triangle}
Suppose that we have a Nisnevich square as in \cref{Definition:Nis_sqaure}, where all vertices are smooth and representable over some base Nis-loc stack $\euS$. Then we have the following cofiber sequence
\begin{center}
    $\mathcal{M}_\euS (\euT) \rightarrow \mathcal{M}_\euS (\euU) \oplus \mathcal{M}_\euS (\euY) \rightarrow \mathcal{M}_\euS (\euX) \overset{[1]}{\rightarrow}$
\end{center}
induced by the cartesian square \cref{Equation:exceptional_nis_square}.
\end{Corollary}

\begin{proof}
One takes $\mathcal{M} := \mathcal{M}_\euS (\euX)$ in \cref{Equation:exceptional_nis_square}. The result then follows by smooth base change \cref{Proposition:Existence_of_motives}. 
\end{proof}

We now explain a descent result for $G$-equivariant resolutions of singularites. The following result is a direct corollary of  \cite{ATW}[Thm. 8.1.2] of Abramovich, Temkin and W\l{}odarczyk and will be critical for the main results of this paper. See also \cite{KollarResSing}[3.4].

\begin{Theorem}
\label{Theorem:Equivariant_Resultion_of_singularities}
Suppose $k$ is of characteristic $0$ and let $ G$ be a smooth group scheme acting on a reduced quasi-projective $X$. Then there exists a $G$-equivariant surjective projective birational morphism $\tilde{X} \to X$ such that $\tilde{X}$ is smooth over $k$. 
\end{Theorem}

\begin{proof}
Consider the action map $ a: G \times X \rightarrow X$ and let $\tilde{X}$ be the resolution of $X$. Then by smooth functoriality of the resolution in \cite{ATW}[Thm. 8.1.2] we see that we have a identification
\begin{center}
    $\widetilde{G \times X} \simeq (G \times X) \times_X \tilde{X}$.
\end{center}
But we also have the projection map $p_X : G \times X \rightarrow X$ which is again smooth. Using \cite{ATW}[Thm. 8.1.2] we have the identification
\begin{center}
$\widetilde{G \times X} \simeq G \times \tilde{X}.$
\end{center}
Putting both identifications together defines a group action $\tilde{a}:  G \times \tilde{X} \rightarrow \tilde{X}$ such that the commutative diagram
\begin{center}
\begin{tikzpicture}
		\node (TL) at (0 , 1.2) {$G \times \tilde{X}$};
		\node (TR) at (2, 1.2) {$\tilde{X}$};
		\node (BL) at (0, 0){$G \times X$};
        \node (BR) at (2,0){$X$};
		
		\draw[->] (TL) -- (TR) node [midway, above] {$\tilde{a}$};
		\draw[->] (TL) -- (BL);
		\draw[->] (TR) -- (BR);
        \draw[->] (BL) -- (BR) node [midway, above] {$a$};
	\end{tikzpicture}
\end{center}
Cartesian. The claim follows.
\end{proof}

In particular \cref{Theorem:Equivariant_Resultion_of_singularities} implies that for a reduced algebraic stack $\euX = [X/G]$ of finite type over a field $k$ of characteristic $0$ there is a regular stack $\tilde{\euX} =  [ \tilde{X}/ G]$ over $k$, together with a projective birational morphism $\tilde{\euX} \rightarrow \euX$. Moreover we have the following:

\begin{Corollary}
    \label{Corollary:Descent_for_resolution_of_singularities}
    Given a finite type reduced algebraic stack $\euX = [X/G]$ over a field $k$ of characteristic $0$, there is a $cdh$-square
    \begin{center}
	\begin{tikzpicture}
		\node (TL) at (0 , 1.2) {$\euT$};
		\node (TR) at (1.2, 1.2) {$\tilde{\euX}$};
		\node (BL) at (0, 0){$\euZ$};
        \node (BR) at (1.2,0){$\euX.$};
		
		\draw[->] (TL) -- (TR) node [midway, above] {$k$};
		\draw[->] (TL) -- (BL)node [midway, left] {$g$};
		\draw[->] (TR) -- (BR)node [midway, right] {$f$};
        \draw[->] (BL) -- (BR) node [midway, below] {$i$};
	\end{tikzpicture}
\end{center}
In particular for any $ \mathcal{M} \in \DM_{\gm} (\euX)$ the associated square \cref{Equation:cdh_square} is cartesian. 
\end{Corollary}

\section{Geometric motives}
\label{Section:Geometric motives}

For this section we will insist on the assumption that the base $B = \spec(k)$ has \emph{characteristic $0$}. Let $ X $ a finite type scheme over $k$. We have the following definition of the category geometric motives.

 \begin{Definition}
\label{Definition:geometric motives for schemes}
Let $\DM_{\gm} (X, \Lambda)$ be the smallest full stable $\infty$-subcategory of $\DM(X, \Lambda)$ which is closed under retracts, and generated by objects
\begin{equation*}
    \{ f_\# (1_Z) (q) \ |\ f: Z\rightarrow X \text{ smooth},\  q \in \mathbf{Z}\}. 
\end{equation*}
We call $\DM_{\gm} (X, \Lambda)$ \emph{The category of geometric motives over $X$}. 
\end{Definition}

Our first observation is that we need not take all smooth morphisms over $X$ when $X$ quasi-projective in \cref{Definition:geometric motives for schemes}. In fact it is enough to take smooth morphisms which are \emph{quasi-projective}.

\begin{Lemma}
\label{Lemma:quasi-projective geometric motives are the same over quasi-projective schemes}
The category of geometric motives over a quasi-projective scheme $X$, can be equivalently described as the  smallest full stable $\infty$-subcategory of $\DM(X, \Lambda)$ which is closed under retracts, and generated by objects
\begin{equation*}
    \{ f_\# (1_Z) (q) \ |\ f: Z\rightarrow X \text{ smooth and quasi-projective},\  q \in \mathbf{Z}\}. 
\end{equation*}
\end{Lemma}

\begin{proof}
First we note that in \cref{Definition:geometric motives for schemes} it is enough to consider smooth morphisms $ f:Z \rightarrow X$ such that $Z$ is quasi-projective over $B$. Indeed, we may cover $Z$ by affines and then inductively use the Mayer-Vietoris sequence. 

Secondly, if $X$ itself is quasi-projective then the claim follows because any smooth morphism between two quasi-projective schemes is itself quasi-projective. 
\end{proof}

\begin{Remark}
\cref{Lemma:quasi-projective geometric motives are the same over quasi-projective schemes} also holds for more generally for finite type schemes $X$. Indeed, For a general (reduced) finite type scheme $X$, we may stratify it by quasi-projective schemes and an induction argument together with the localization triangle suffices to show this.
\end{Remark}

For our main application, we will be interested in stacks $[X/G]$ which are quasi-projective over $BG$, this together with \cref{Lemma:quasi-projective geometric motives are the same over quasi-projective schemes} motivates our definition of geometric motives over a stack.

\begin{Definition}
\label{Definition:gemoetric motives}
Let $\euX$ be in $\nislocst$, we define $\DM_{\gm} (\euX, \Lambda)$ to be the smallest full stable subcategory of $\DM(\euX, \Lambda)$ which is closed under retracts, and generated by objects
\begin{equation*}
    \{ M_\euX (\euZ)(q) \ |\ \euZ \text{  smooth, representable and quasi-projective over } \euX,\  q \in \mathbf{Z}\}
\end{equation*}
where $ M_\euX (\euZ) : = f_\# (1_\euZ) $ for $ f: \euZ \rightarrow \euX$. 
\end{Definition}

\begin{Remark}
The perhaps more natural variant of \cref{Definition:gemoetric motives} where the quasi-projective condition is omitted could also be considered. We think it would be interesting to compare these two notions. Ultimately, we chose to use \cref{Definition:gemoetric motives} because it was easier to show that it was well-behaved under $*$-pushforwards by closed immersions.
\end{Remark}

In this section we will first establish various results about when the $\infty$-subcategory $\DM_{\gm} (\euX, \Lambda)$ of geometric motives in $\DM (\euX, \Lambda)$, is preserved under some of the six operations.

Once we have addressed this, we will use these results in the second part of the section to show that the category of geometric motives can be generated by a smaller subcategory whose mapping spectra have good connectedness properties. In particular we will see that in case $\euX$ is a quotient of a quasi-projective scheme by $G$ over a field of characteristic $0$, the category of geometric motives can be generated by the subcategory of so called Chow motives.

Before we begin we would also like to point out that that while it is very natural to wonder if the objects in $\DM_{\gm} (\euX, \Lambda)$ are compact in $\DM_{\gm} (\euX, \Lambda)$, this is in general not the case. In fact, the example of \cite{ho2023revisiting}[Ex 4.1.12] in the setting of constructible sheaves works in our setting as well:

\begin{Example}
\label{Example:Geometric motives are not compact for stacks}
Let $\euX = B \mathbf{G}_m$, we claim that $ 1_{B \mathbf{G}_m}$ is not compact in $\DM(B\mathbf{G}_m, \Lambda)$. 

First by \cite{RicharzScholbach}[Thm. 2.2.10] we can identify
\begin{center}
    $\pi_0 \map_{\DM(B\mathbf{G}_m, \Lambda)} (1, 1(*)[2*]) \simeq \chow^* (B \mathbf{G}_m)_\Lambda \simeq \Lambda [x]$
\end{center}
where $x$ is in degree 1. In particular $x$ lifts to a map
\begin{center}
   $ x: 1_{B\mathbf{G}_m} \rightarrow  1_{B\mathbf{G}_m} (1)[2]$. 
\end{center}
in $\DM(B\mathbf{G}_m, \Lambda)$. We note that if we pull back $x$ along the covering morphism $\pi: \text{pt}  \rightarrow B\mathbf{G}_m$, we have that $\pi^* x \simeq 0$, because $ x $ corresponds to $ c_1 (\mathcal{O} (1))$. 

Thus if we pullback the filtered colimit
\begin{center}
    $\mathcal{M} := \varinjlim  (1_{B\mathbf{G}_m} \overset{x}{\rightarrow}  1_{B\mathbf{G}_m} (1)[2] \overset{x}{\rightarrow} 1_{B\mathbf{G}_m} (2)[4] \overset{x}{\rightarrow} \cdots)$
\end{center}
along $\pi$ we see that 
\begin{center}
    $\pi^* \mathcal{M} \simeq 0,$
\end{center}
and since $\pi$ is a smooth cover, it follows that $\pi^*$ is conservative hence $\mathcal{M} \simeq 0$. It follows that 
\begin{center}
    $\map_{\DM(B\mathbf{G}_m, \Lambda)} (1, \mathcal{M}) \simeq 0.$
\end{center}
On the other hand 
\begin{center}
   $ \varinjlim (\map_{\DM(B\mathbf{G}_m, \Lambda)} (1, 1) \rightarrow \map_{\DM(B\mathbf{G}_m, \Lambda)} (1, 1(1) [2])\rightarrow \cdots)$
\end{center}
can be seen to be not equivalent to $0$ by looking at for instance $\pi_0$ which is $\Lambda[x, x^{-1}]$. 
\end{Example}

The upshot of \cref{Example:Geometric motives are not compact for stacks}, is that is shows that the geometric motives are in general not compact objects in $\DM(\euX, \Lambda)$. 

\subsection{The six operations and geometric motives}

\begin{Lemma}
\label{Lemma:DM_gm_is_stable_under_upper_*}
For any $ f: \euX \rightarrow \euY$ of stacks, the functor $f^*$ restricts to a functor
\begin{center}
$f^*:\DM_{\gm} (\euY, \Lambda) \rightarrow \DM_{\gm}(\euX, \Lambda)$
\end{center}
\end{Lemma}

\begin{proof}
To prove the claim it is enough to check on generators $\mathcal{M}_\euY(\euZ)(q)$ in $\DM_{\gm}(\euY, \Lambda)$ thus by considering the cartesian square cartesian square
\begin{center}
	\begin{tikzpicture}
		\node (TL) at (0 , 1.5) {$\euZ \times_\euY \euX$};
		\node (TR) at (1.5, 1.5) {$ \euZ$};
		\node (BL) at (0, 0){$\euX$};
        \node (BR) at (1.5,0){$\euY$};
		
		\draw[->] (TL) -- (TR) node [midway, above] {$g$};
		\draw[->] (TL) -- (BL)node [midway, left] {$q$};
		\draw[->] (TR) -- (BR)node [midway, right] {$p$};
        \draw[->] (BL) -- (BR) node [midway, below] {$f$};
	\end{tikzpicture}
\end{center}
where $ \mathcal{M}_{\euY} (\euZ) (q)  = p_\# 1_\euZ (q) $, from \cref{Proposition:Existence_of_motives} it follows that
\begin{equation*}
     f^*(\mathcal{M}_\euY(\euZ)(q)) \simeq f^*p_{\#}1_{\euZ} (q)  \simeq q_{\#}g^*1_\euZ (q).
\end{equation*}
\end{proof}

\begin{Proposition}
\label{Lemma:DM_gm is stable under lower ! for smooth maps}
If $ f: \euX \rightarrow \euY$ is smooth and representable then the functor $ f_\#$ restricts to a functor 
\begin{center}
    $f_\#:\DM_{\gm} (\euX, \Lambda) \rightarrow \DM_{\gm}(\euY, \Lambda)$.
\end{center}
\end{Proposition}

\begin{proof}
This follows from the fact that for a smooth representable $g: \euZ \rightarrow \euX$, we have that $f_\# \circ g_\# \simeq (f \circ g)_\#$. 
\end{proof}

\begin{Lemma}
\label{Lemma:DM_gm_is_closed_under_tensor_product}
If $ \mathcal{M}, {N} \in \DM_{\gm} (\euX, \Lambda)$ then so is $ \mathcal{M} \otimes \mathcal{N}.$
\end{Lemma}

\begin{proof}
    The proof is the same as in \cite{CiDeg}[4.2.3].
\end{proof}

\begin{Lemma}
\label{Lemma:DM_gm_is_zariski_local}
Suppose $\euX$ is a finite type Nis-loc stack
and that there exists a Zariski cover $ \euX = \bigcup_i \euU_i$, then an object  $ \mathcal{M} \in \DM(\euX, \Lambda)$ is in $\DM_{\gm}(\euX, \Lambda)$ if and only if $M|_{\euU_i}$ is in $\DM_{\gm}(\euU_i, \Lambda)$. 
\end{Lemma}

\begin{proof}
By arguing inductively it is enough to consider the case that $\euX = \euU \cup \euV$. Via the Nisnevich square
\begin{center}
	\begin{tikzpicture}
		\node (TL) at (0 , 1.5) {$\euU \times_\euX \euV$};
		\node (TR) at (1.5, 1.5) {$ \euV$};
		\node (BL) at (0, 0){$\euU$};
        \node (BR) at (1.5,0){$\euX,$};
		
		\draw[right hook->] (TL) -- (TR) node [midway, above] {$j_{\euU}'$};
		\draw[right hook->] (TL) -- (BL)node [midway, left] {$j'_\euV$};
		\draw[right hook->] (TR) -- (BR)node [midway, right] {$j_\euV$};
        \draw[right hook->] (BL) -- (BR) node [midway, below] {$j_\euU$};
	\end{tikzpicture}
\end{center}
for each $\mathcal{M} \in \DM(\euX, \Lambda)$, write $ \mathcal{M}_{\euW} := {j_\euW}_{\#} j^*_\euW \mathcal{M} $ for $\euW= \euU \times_\euX \euV, \euU, \euV$,  we get a triangle of motives by \cref{Corollary:Mayer-Vietoris_squares}
\begin{equation*}
    \mathcal{M}_{\euU \times_\euX \euV} \rightarrow \mathcal{M}_{\euU} \oplus \mathcal{M}_{\euV} \rightarrow \mathcal{M} \overset{[1]}{\rightarrow }
\end{equation*}
Then since $\mathcal{M}_{\euU \times_\euX \euV}$ and $\mathcal{M}_{\euU} \oplus \mathcal{M}_{\euV}$ are contained in $\DM_{\gm} (\euX, \Lambda)$ it follows that $\mathcal{M} \in \DM_{\gm}(\euX, \Lambda)$. 
\end{proof}

\begin{Lemma}
\label{Lemma:Thom_spaces_preserve_DM_gm}
For any stack $\euX$ and vector bundle $\eupsilon$ over $\euX$, tensoring by $Th(\eupsilon)$ and $Th(-\eupsilon)$ preserve $\DM_{\gm}(\euX, \Lambda)$. 
\end{Lemma}

\begin{proof}
This follows from the fact that $\DM(\euX, \Lambda)$ is oriented i.e. $ - \otimes Th(\eupsilon) \simeq (n) [2n]$ .
\end{proof}

\begin{Corollary}
\label{Lemma:DM_gm is stable under lower-* for smooth proper representable morphisms}
Let $f: \euX \rightarrow \euY$ be a smooth and proper representable morphism in $\nislocst$. Then the functor $f_*$ restricts to a functor
\begin{center}
$f_* : \DM_{\gm} (\euX, \Lambda) \rightarrow \DM_{\gm} (\euY, \Lambda).$
\end{center}
\end{Corollary}

\begin{proof}
The corollary follows immediately from \cref{Lemma:DM_gm is stable under lower ! for smooth maps}, the equivalence $\alpha_f: f_!  \simeq f_*$ of \cref{Proposition:Existence_of_exceptional_functors} (3) and purity \cref{Proposition:purity} together with \cref{Lemma:Thom_spaces_preserve_DM_gm}. 
\end{proof}

Our goal now is to show that for projective morphisms the lower-$*$ functor preserves geometric objects. We will do this in two steps, first we will show that closed immersions have this property and then use the fact that for a general projective morphism we can factor it by closed immersion and a morphism from a projective space.

\begin{Lemma}
\label{Lemma:closed immersions preserve geometric objects special case}
Let $i: \euX \hookrightarrow \euY$ be a closed immersion in $\nislocst$ and suppose that $\euX$ has the resolution property. Then the functor $i_*$ restricts to a functor 
\begin{center}
    $i_*: \DM_{\gm} (\euX, \Lambda) \rightarrow \DM_{\gm} (\euY, \Lambda).$
\end{center}
\end{Lemma}

\begin{proof}
Let $f_0: \euZ_0 \rightarrow \euX$ be a smooth quasi-projective representable morphism over $\euX$. First we assume that $\euZ_0$ is linearly fundamental in the sense of \cite{alper2023etale}[Def 2.7]. Now we apply \cite{AHLHR}[Thm. 1.3] so that we have a cartesian square
\begin{center}
	\begin{tikzpicture}
		\node (TL) at (0 , 1.5) {$\euZ_0$};
		\node (TR) at (1.5, 1.5) {$ \euZ$};
		\node (BL) at (0, 0){$\euX$};
        \node (BR) at (1.5,0){$\euY$};
		
		\draw[right hook ->] (TL) -- (TR) node [midway, above] {$i'$};
		\draw[->] (TL) -- (BL)node [midway, left] {$f_0$};
		\draw[->] (TR) -- (BR)node [midway, right] {$f$};
        \draw[right hook ->] (BL) -- (BR) node [midway, below] {$i$};
	\end{tikzpicture}
\end{center}
where $f$ is smooth and representable. Let $\euU := \euX - \euY$, and consider the localization diagram 

\begin{center}
	\begin{tikzpicture}
		\node (TL) at (0 , 1.5) {$\euZ_0$};
		\node (TM) at (1.5, 1.5) {$ \euZ$};
        \node (TR) at (3, 1.5) {$\euZ_\euU$};
		\node (BL) at (0, 0){$\euX$};
        \node (BM) at (1.5,0){$\euY$};
        \node (BR) at (3, 0) {$\euU$};
		
		\draw[right hook ->] (TL) -- (TM) node [midway, above] {$i'$};
		\draw[->] (TL) -- (BL) node [midway, left] {$f_0$};
		\draw[->] (TM) -- (BM) node [midway, right] {$f$};
        \draw[right hook ->] (BL) -- (BM) node [midway, below] {$i$};
        \draw [left hook ->] (BR) -- (BM) node [midway, below] {$j$};
        \draw [->] (TR) --  (BR) node [midway, right] {$f_\euU$};
        \draw [left hook ->] (TR) -- (TM) node [midway, above] {$j'$};
        
	\end{tikzpicture}
\end{center}
which gives rise to a cofiber sequence
\begin{center}
 $\mathcal{M}_\euY (\euZ_\euU) \rightarrow \mathcal{M}_\euY (\euZ) \rightarrow i_* \mathcal{M}_\euX (\euZ_0) \overset{[1]}{ \rightarrow .}$
\end{center}
Since both $\mathcal{M}_\euY (\euZ_\euU)$ and $\mathcal{M}_\euY (\euZ)$ are geometric this implies that $i_* \mathcal{M}_\euX (\euZ_0)$ is geometric which is what we wanted to show. In the general case, for a smooth quasi-projective morphism $f: \euZ_0  \rightarrow \euX$, since $\euX$ has the resolution property, we know that $\euZ_0 \simeq [Z_0 /GL_n]$ for some quasi-projective $Z_0$. Now applying the $G$-equivariant Jouanolou's trick \cite{HoyoisSix}[Prop. 2.20] we get an affine bundle
\begin{center}
   $ \pi: \tilde{\euZ}_0 \rightarrow \euZ_0$
\end{center}
where $\tilde{\euZ}_0$ is linearly fundamental. But since $\pi^*$ is fully-faithful the counit $\pi_\# \pi^* \overset{\simeq}{\rightarrow} \id$ is an equivalence and thus
\begin{center}
$\mathcal{M}_{\euX} (\tilde{\euZ}_0) \simeq f_{0, \#} \pi_{\#} \pi^* (1_{\euZ_0}) \simeq \mathcal{M}_{\euX} (\euZ_0).$
\end{center}
But since $\tilde{\euZ}_0$ is linearly fundamental by the previous case $i_* \mathcal{M}_{\euX} (\tilde{\euZ}_0)  \in \DM_{\gm} (\euY, \Lambda)$ and the claim now follows.
\end{proof}

\begin{Proposition}
\label{Proposition:closed immersions preserve geometric objects general case}
Let $\iota: \euX \hookrightarrow \euY$ be a closed immersion in $\nislocst$. Then the functor $i_*$ restricts to a functor 
\begin{center}
    $\iota_*: \DM_{\gm} (\euX, \Lambda) \rightarrow \DM_{\gm} (\euY, \Lambda).$
\end{center}
\end{Proposition}

\begin{proof}
By reduced invariance we may assume that $\euY$ is reduced. Thus, by \cite{HallRydhStratification} there is a stratification of $\euY$ be stacks with the resolution property. We proceed by induction on the length of the stratification. In the case that $\euY$ itself has the resolution property it follows that since $\iota :\euX \rightarrow \euY$ is a closed immersion, $\euX$ also has the resolution property and the claim follows by \cref{Lemma:closed immersions preserve geometric objects special case}. In the general case assume that we have a stratification of length $n$:
\begin{center}
    $\emptyset = \euY_0 \hookrightarrow \euY_1 \hookrightarrow \cdots \hookrightarrow \euY_n = \euY$
\end{center}
Let $\euU: = \euY - \euY_{n-1} $ and consider the localization square
\begin{center}
	\begin{tikzpicture}
		\node (TL) at (0 , 1.5) {$\euX_{n-1}$};
		\node (TM) at (1.5, 1.5) {$ \euX$};
        \node (TR) at (3, 1.5) {$\euX_\euU$};
		\node (BL) at (0, 0){$\euY_{n-1}$};
        \node (BM) at (1.5,0){$\euY$};
        \node (BR) at (3, 0) {$\euU.$};
		
		\draw[right hook ->] (TL) -- (TM) node [midway, above] {$i'$};
		\draw[->] (TL) -- (BL) node [midway, left] {$\iota_{n-1}$};
		\draw[->] (TM) -- (BM) node [midway, right] {$\iota$};
        \draw[right hook ->] (BL) -- (BM) node [midway, below] {$i$};
        \draw [left hook ->] (BR) -- (BM) node [midway, below] {$j$};
        \draw [->] (TR) --  (BR) node [midway, right] {$\iota_\euU$};
        \draw [left hook ->] (TR) -- (TM) node [midway, above] {$j'$};
	\end{tikzpicture}
 \end{center}
For $\mathcal{M} \in \DM_{\gm} (\euX, \Lambda)$ we have a  cofiber sequence
 \begin{center}
 $j_{\#} \iota_{\euU, *} j'^* (\mathcal{M}) \rightarrow \iota_* (\mathcal{M}) \rightarrow i_* \iota_{n-1, *} i'^* (\mathcal{M}) \overset{[1]}{ \rightarrow }.$
\end{center}
We now observe that $j_{\#} i_{\euU, *} j'^* (\mathcal{M}) \in \DM_{\gm} (\euY, \Lambda)$ by \cref{Lemma:DM_gm_is_stable_under_upper_*}, \cref{Lemma:closed immersions preserve geometric objects special case} and \cref{Lemma:DM_gm is stable under lower ! for smooth maps} and $i_* \iota_{n-1, *} i'^* (\mathcal{M}) \in \DM_{\gm} (\euY, \Lambda)$ by \cref{Lemma:DM_gm_is_stable_under_upper_*}, induction and, \cref{Lemma:closed immersions preserve geometric objects special case}. Thus $\iota_* \mathcal{M} \in \DM_{\gm} (\euY, \Lambda)$. 
\end{proof}

\begin{Lemma}
\label{Lemma:DM_gm is preserved by lower * for projective morphisms special case}
Let $f: \euX \rightarrow \euY $ be a projective morphism in $\nislocst$. Suppose that $\euY$ has the resolution property. Then the functor $f_*$ restricts to a functor
\begin{center}
$f_* : \DM_{\gm} (\euX, \Lambda) \rightarrow \DM_{\gm} (\euY, \Lambda)$. 
\end{center}
\end{Lemma}

\begin{proof}
Since $\euY$ has the resolution property we may factor $f:\euX \rightarrow \euY$ as
\begin{center}
$\euX \overset{\iota}{\rightarrow} \euP \overset{p}{\rightarrow} \euY $
\end{center}
where $\iota$ is a closed immersion and $p$ is a smooth representable proper morphism.  The claim now follows from \cref{Proposition:closed immersions preserve geometric objects general case} and \cref{Lemma:DM_gm is stable under lower-* for smooth proper representable morphisms}.
\end{proof}

\begin{Proposition}
\label{Proposition:DM_gm is preserved by lower * for projective morphisms general case}
Let $f: \euX \rightarrow \euY $ be a projective morphism in $\nislocst$. Then the functor $f_*$ restricts to a functor
\begin{center}
$f_* : \DM_{\gm} (\euX, \Lambda) \rightarrow \DM_{\gm} (\euY, \Lambda)$. 
\end{center}
\end{Proposition}

\begin{proof}
By reduced invariance we may assume that $\euY$ is reduced. Thus, by \cite{HallRydhStratification} there is a stratification of $\euY$ be stacks with the resolution property. We proceed by induction on the length of the stratification. In the case that $\euY$ has the resolution property we are done by \cref{Lemma:DM_gm is preserved by lower * for projective morphisms special case}. In the general case assume that we have a stratification of length $n$:
\begin{center}
    $\emptyset = \euY_0 \hookrightarrow \euY_1 \hookrightarrow \cdots \hookrightarrow \euY_n = \euY$
\end{center}
Let $\euU: = \euY - \euY_{n-1} $ and consider the localization square
\begin{center}
	\begin{tikzpicture}
		\node (TL) at (0 , 1.5) {$\euX_{n-1}$};
		\node (TM) at (1.5, 1.5) {$ \euX$};
        \node (TR) at (3, 1.5) {$\euX_\euU$};
		\node (BL) at (0, 0){$\euY_{n-1}$};
        \node (BM) at (1.5,0){$\euY$};
        \node (BR) at (3, 0) {$\euU.$};
		
		\draw[right hook ->] (TL) -- (TM) node [midway, above] {$i'$};
		\draw[->] (TL) -- (BL) node [midway, left] {$f_{n-1}$};
		\draw[->] (TM) -- (BM) node [midway, right] {$f$};
        \draw[right hook ->] (BL) -- (BM) node [midway, below] {$i$};
        \draw [left hook ->] (BR) -- (BM) node [midway, below] {$j$};
        \draw [->] (TR) --  (BR) node [midway, right] {$f_\euU$};
        \draw [left hook ->] (TR) -- (TM) node [midway, above] {$j'$};
	\end{tikzpicture}
 \end{center}
For $\mathcal{M} \in \DM_{\gm} (\euX, \Lambda)$ we have a  cofiber sequence
 \begin{center}
 $j_{\#} f_{\euU, *} j'^* (\mathcal{M}) \rightarrow f_* (\mathcal{M}) \rightarrow i_* f_{n-1, *} i'^* (\mathcal{M}) \overset{[1]}{ \rightarrow }.$
\end{center}
Now, $j_{\#} f_{\euU, *} j'^* (\mathcal{M}) \in \DM_{\gm} (\euY, \Lambda)$ by \cref{Lemma:DM_gm is preserved by lower * for projective morphisms special case} and $i_* f_{n-1, *} i'^* (\mathcal{M}) \in \DM_{\gm} (\euY, \Lambda)$ by induction. Thus we see that $f_* (\mathcal{M}) \in \DM_{\gm} (\euY, \Lambda)$ which is what we wanted to show.
\end{proof}

\begin{Remark}
With an appropriate form of Chow's lemma for stacks, one can extend \cref{Proposition:DM_gm is preserved by lower * for projective morphisms general case} to proper representable morphisms. 
\end{Remark}

\begin{Corollary}
\label{Corollary:qproj_lower_shriek_preserves_geometric_motives}
    Suppose $f: \euX \rightarrow \euY$ is quasi-projective morphism in $\nislocst$. Then the functor $f_!$ restricts to a functor
    \begin{center}
        $f_! : \DM_{\gm}(\euX, \Lambda) \rightarrow \DM_{\gm}(\euY, \Lambda)$. 
    \end{center}
\end{Corollary}

\begin{proof}
Since $f: \euX \rightarrow \euY$ is quasi-projective, we may factor $f$ as
\begin{center}
$\euX \overset{j}{\hookrightarrow} \EuScript{P} \overset{p}{\rightarrow }\euY$
\end{center}
where $j$ is an open immersion and $p$ is projective. The result now follows from \cref{Lemma:DM_gm is stable under lower ! for smooth maps} and \cref{Proposition:DM_gm is preserved by lower * for projective morphisms general case}.
\end{proof}

\subsection{Generation results for the derived category of geometric motives}

\begin{Definition} 
\label{Definition:Infty category of chow motives}
For an $\euX \in \nislocst$ we define the additive $\infty$-category of Chow motives
\begin{center}
$\CM(\euX, \Lambda) \subseteq \DM (\euX, \Lambda)$
\end{center}
to be smallest additive $\infty$-category generated by 
\begin{center}
$\{ f_! 1_\euZ(q)[2q]:\ \euZ \text{ smooth over } B ,\  f:\euZ \rightarrow \euX \text{ projective}, q \in \mathbf{Z} \}$.
\end{center}
and retracts thereof. 
\end{Definition}

In this section we wish to prove that $\CM (\euX)$ generates $\DM_{\gm} (\euX, \Lambda)$ under finite limits, colimits and retracts. i.e. that the smallest thick subcategory of $\DM_{\gm}(\euX, \Lambda)$ containing $\CM(\euX)$ is $\DM_{\gm}(\euX, \Lambda)$. For the readers familiar with \cite{CiDeg} we remark that our strategy was inspired by and will follow closely that of \cite{CiDeg}[4.4.3]. We start first with an elementary lemma about stacks, which will be useful for induction arguments.

\begin{Lemma}
\label{Lemma:dimension_of_complement_of_dense_open_drops}
Let $\euX$ be a finite type algebraic stack over $B$. Suppose that $\euU \subseteq \euX$ is a dense open substack of $\euX$. Let $\euZ$ denote the complement of $\euU$ in $\euX$ Then $\dim(\euZ) < \dim(\euX)$. 
\end{Lemma}

\begin{proof}
Consider the diagram of cartesian squares where $\Pi$ is a smooth cover
\begin{center}
	\begin{tikzpicture}
	   \node (TL) at (0 , 1.5) {$ U$};
	   \node (TM) at (2, 1.5) {$X$};
        \node (TR) at (4, 1.5) {$Z$};
	   \node (BL) at (0, 0) {$\euU$};
       \node (BM) at (2,0) {$\euX$};
       \node (BR) at (4, 0) {$\euZ.$};

       \draw[->] (TL) -- (TM) node [midway, above] {$j'$};
	   \draw[->] (TR) -- (TM) node [midway, above] {$i'$};
	   \draw[->] (BL) -- (BM) node [midway, below] {$j$};
        \draw[->] (BR) -- (BM) node [midway, below] {$i$};
	   \draw[->] (TM) -- (BM) node [midway, right] {$\pi$};
	   \draw[->] (TL) -- (BL)node [midway, left] {$\pi_\euU$};
	   \draw[->] (TR) -- (BR)node [midway, right] {$\pi_\euZ$};
\end{tikzpicture}
\end{center}
It follows that since the map $\pi$ is a continuous and surjective on underlying topological spaces that $U$ is a dense open subscheme of $X$ with complement $Z$. 

We will be finished if we show for each $z \in |\euZ|$ that 
\begin{center}
    $\dim_z (\euZ) < \dim_z (\euX)$ 
\end{center}

We are free to pick any lift of the point $z$ to $Z$. In particular there exists a lift $\tilde{z}$ such that $\dim_{\tilde{z}} (Z) = \dim(Z)$ and we have that 
\begin{center}
    $\dim_{z} (\euZ) = \dim(Z) - \dim(R_{{\euZ},z}).$
\end{center}
Similarly we may pick any lift of $z$ in $X$ such that
\begin{center}
    $\dim_{z} (\euX) = \dim(X) - \dim(R_{{\euX},z}).$
\end{center}
For each $ z \in |\euZ|$ we have a canonical equivlance $ R_{\euZ,z} \simeq R_{z}$ where $ R : = X \times_\euX X$ and $R_\euZ$ is the restriction to $\euZ$. But since $U$ is dense in $X$ with complement $Z$ we have that $\dim(Z) < \dim(X)$ and hence
\begin{center}
    $\dim_z (\euZ) < \dim_z (\euX)$,
\end{center}
which is what we wanted to show.
\end{proof}

\begin{Proposition}
\label{Proposition:seperated_lower_shriek_preserves_DM_gm}
Suppose that $f:\euX \rightarrow \euY$ is representable and separated and $\euY$ has the resolution property. Then $f_!$ restricts to a functor
\begin{center}
   $ f_!: \DM_{\gm}(\euX, \Lambda) \rightarrow \DM_{\gm} (\euY, \Lambda)$
\end{center}
\end{Proposition}

\begin{proof}
Since $\euY$ has the resolution property it is of the form $[Y/GL_n]$ where $Y$ is quasi-affine. Since $ f:\euX \rightarrow \euY$ is representable it follows that $\euX \simeq [X/GL_n]$ for $X$ an algebraic space.  By reduced invariance we may assume that $\euX$ is reduced.

The stack $\euX$ has affine stabilizers is finite type and quasi-separated, and by \cite{HallRydhStratification}[Prop. 2.6] there exists a stratification of $\euX$ by global quotient stacks which are quasi-projective over $BGL_n$. We will use induction on the length of the stratification. In the trivial case when $ \euX$ is quasi-projective over $BGL_n$ then $f:\euX \rightarrow \euY$ is quasi-projective and we are done by \cref{Corollary:qproj_lower_shriek_preserves_geometric_motives}. 

For a stratification of length $n$
\begin{center}
  $ \emptyset= \euX_0 \hookrightarrow \euX_1 \hookrightarrow \cdots \hookrightarrow \euX_n = \euX$
\end{center}
we consider the diagram
\begin{center}
   $\euX_{n-1} \overset{i}{\hookrightarrow} \euX \overset{j}{\hookleftarrow} \euU$
\end{center}
of stacks over $\euY$. Since $\euU$ and $\euY$ are quasi-projective over $BGL_n$, it follows that the induced map $ f|_{\euU}: \euU \rightarrow \euY$ is quasi-projective. Let $\mathcal{M} \in \DM_{\gm} (\euX, \Lambda)$ and consider the localization triangle induced by \cref{Proposition:Localization}
\begin{center}
    $j_! j^! (\mathcal{M}) \rightarrow \mathcal{M} \rightarrow i_* i^* \mathcal{M} \overset{[1]}{\rightarrow}.$
\end{center}
Since $f_!$ is an exact functor between stable $\infty$-categories we get a cofiber sequence

 \begin{center}
    $f_{\euU ,!} j^* (\mathcal{M}) \rightarrow f_! \mathcal{M} \rightarrow f_{n-1, !}  i^* \mathcal{M} \overset{[1]}{\rightarrow}.$
\end{center}

Now,\cref{Corollary:qproj_lower_shriek_preserves_geometric_motives} implies that $f_{\euU ,!} j^* (\mathcal{M}) \in \DM_{gm} (\euY, \Lambda)$ and induction implies that $f_{n-1, !}  i^* \mathcal{M} \in \DM_{\gm} (\euY, \Lambda)$. Thus $f_! (\mathcal{M}) \in \DM_{\gm} (\euY, \Lambda)$ finishing the argument. 
\end{proof} 

Recall, that a full subcategory $\mathscr{D}$ of a stable $\infty$-category $\mathscr{C}$ is called \emph{thick} if it is closed under taking retracts (see \cite{higher_topos_theory} [4.4.5] for a discussion on retracts and idempotents in the setting of $\infty$-categories). 

\begin{Theorem}
\label{Theorem:DM_gm_is_generated_by_projective_morphisms}
Let $\euX$ be a finite type Nis-loc stack with affine stabilizers. The category $\DM_{\gm}(\euX)$ is the smallest thick stable $\infty$-subcategory of $\DM (\euX, \Lambda)$ generated by the collection of objects
\begin{center}
    $\mathscr{P}(\euX):= \{ f_! (1_{\euX'} (n))\ | \  f: \euX' \rightarrow \euX \text{ is projective and }  n \in \mathbf{Z} \}$. 
\end{center}
\end{Theorem}

\begin{proof}
Let $\DM_{\text{proj}}(\euX)$ be the smallest thick subcategory generated by $\mathscr{P}(\euX)$. By \cref{Proposition:DM_gm is preserved by lower * for projective morphisms general case} it follows that $\DM_{\text{proj}}(\euX) \subset \DM_{\gm}(\euX)$. So we prove the reverse inclusion. 

For any quasi-projective smooth morphism $ f: \euX' \rightarrow \euX$ it follows from purity that  $ f_\#$ agrees with $ f_!$ up to a Tate twist. Thus it is enough to prove that   $ f_{!} 1_{\euX'} $ for any such $f$ is contained in $\DM_{\text{proj}}(\euX)$.

When $\euX$ has the resolution property we are finished by \cref{Proposition:seperated_lower_shriek_preserves_DM_gm}. In the general case, we can argue by induction on the length of the stratification of $\euX$ by stacks with the resolution property. Note that we may assume that $\euX$ is reduced by reduced invariance. That is for a length $n$ stratification
\begin{center}
$ \emptyset= \euX_0 \hookrightarrow \euX_1 \hookrightarrow \cdots \hookrightarrow \euX_n = \euX$
\end{center}
we consider the diagram
\begin{center}
\begin{tikzpicture}
	   \node (TL) at (0 , 1.5) {$ \euX'_{n-1}$};
	   \node (TM) at (1.8, 1.5) {$\euX'$};
        \node (TR) at (3.6, 1.5) {$\euU'$};
	   \node (BL) at (0, 0) {$\euX_{n-1}$};
       \node (BM) at (1.8,0) {$\euX$};
       \node (BR) at (3.6, 0) {$\euU$};

       \draw[->] (TL) -- (TM) node [midway, above] {$i'$};
	   \draw[->] (TR) -- (TM) node [midway, above] {$j'$};
	   \draw[->] (BL) -- (BM) node [midway, below] {$i$};
        \draw[->] (BR) -- (BM) node [midway, below] {$j$};
	   \draw[->] (TM) -- (BM) node [midway, right] {$f$};
	   \draw[->] (TL) -- (BL)node [midway, left] {$f_{n-1}$};
	   \draw[->] (TR) -- (BR)node [midway, right] {$f_\euU$};
\end{tikzpicture}
\end{center}
of cartesian squares where $i$ is a closed immersion and $j$ is an open immersion. By considering the localization triangle from \cref{Proposition:Localization}
\begin{center}
    $j_! j^! f_! (1_{\euX'}) \rightarrow f_! (1_{\euX'}) \rightarrow i_* i^* f_! (1_{\euX'})  $
\end{center}
it follows from the base change isomorphisms of \cref{Proposition:Existence_of_exceptional_functors}, \cref{Corollary:qproj_lower_shriek_preserves_geometric_motives}, and \cref{Proposition:DM_gm is preserved by lower * for projective morphisms general case} that both $j_! j^! f_! (1_{X'})$ and $i_* i^* f_! (1_{X'})$ are contained in $\DM_{\gm}(\euX, \Lambda)$. Hence we see that $f_! (1_{\euX'})$ is contained in $\DM_{\gm}(\euX, \Lambda)$.
\end{proof}

The next result says that when $\euX$ is a Nis-loc stack which is quasi-projective over $BG$ then $\DM_{\gm} (\euX, \Lambda)$ is the thick closure of $\CM(\euX, \Lambda)$. 

\begin{Theorem}
\label{Theorem:Generation_of_DM_gm_by_chow}
Suppose that $\euX = [X/G]$ where $X$ is a quasi-projective scheme over $B$ and $G$ is affine algebraic group. Then the category $ \CM(\euX, \Lambda)$ generates $\DM_{\gm} (\euX, \Lambda)$ under finite limits, colimits and retracts. 
\end{Theorem}

\begin{proof}

Let $\DM_{\CM} (\euX, \Lambda)$ be the smallest thick subcategory of $\DM (\euX, \Lambda)$ which contains the category $\CM(\euX, \Lambda)$. We must show that  $\DM_{\CM} (\euX, \Lambda)$ is precisely all of  $\DM_{\gm} (\euX, \Lambda)$. To this end by \cref{Theorem:DM_gm_is_generated_by_projective_morphisms}, it will be enough to show that $\mathscr{P}(\euX)$ is contained in in $\DM_{\CM} (\euX, \Lambda)$. 

Consider a projective morphism $f: \euX' \rightarrow \euX$. From our hypothesis on $\euX$  we may take $ \euX' \simeq [X'/G]$ where $X'$ is projective over $X$, moreover without loss of generality we may assume that $\euX'$ is reduced. We now proceed by induction on the relative dimension of $\euX' \rightarrow BG$ the claim is clear when relative dimension is $0$, so we assume it holds for some $n >0$ and consider $\euX' \rightarrow BG$ with relative dimension $n+1$.

 We may apply equivariant resolutions of singularities over $k$, \cref{Theorem:Equivariant_Resultion_of_singularities}, to $ X'$. Thus after taking stack quotients by $G$ we arrive at a projective birational morphism $ [\tilde{X'} / G] \rightarrow [X/G]$. The stack $[\tilde{X'} / G]$ is smooth over $k$. Now we consider the cartesian squares
\begin{center}
	\begin{tikzpicture}
		\node (TL) at (0 , 1.5) {$ \euU' $};
		\node (TM) at (2, 1.5) {$ [\tilde{X'} / G] $};
        \node (TR) at (4,1.5) {$\euZ'$};
		\node (BL) at (0, 0) {$\euU$};
        \node (BM) at (2, 0) {$[X'/G] $};
        \node (BR) at (4,0) {$\euZ$};
		
		\draw[left hook->] (BR) -- (BM);
        \draw[left hook->] (TR) -- (TM);
        \draw[->] (TL) -- (BL);
        \draw[->] (TM) -- (BM);
		\draw[->] (TL) -- (TM);
        \draw[->] (BL) -- (BM);
		\draw[->] (TR) -- (BR);
	\end{tikzpicture}
\end{center}
where $\euU$ is dense open, the right hand side is a $cdh$-square. In particular since the relative dimension of $ \euZ \rightarrow BG$ is strictly less then the relative dimension of $\euX'$ over $BG$ and we may apply induction hypothesis  together with the cofiber sequence induced by \cref{Corollary:Descent_for_resolution_of_singularities}.
\end{proof}

\section{Mapping spectra and Chow groups}
\label{Section:Mapping spectra and chow motives}

 In this section we will study the mapping spectra in $\DM(\euX, \Lambda)$. Our main goal will be to show that the mapping spectra of $\CM(\euX, \Lambda)$ are connective but along the way we will identify the Borel-Moore homology of a quotient stack with the equivariant higher Chow groups. In this section we will take $k$ to be a field of arbitrary characteristic, also in many proofs we will suppress the notation for the coefficient ring and often write $\DM(\euX)$ with the hope of making things easier to read.

 Let $X$ be a quasi-projective scheme over $B$ of dimension $n$ equipped with an action of an affine algebraic group $G$ and integers $ s, t \in \mathbf{Z}$. We fix a Totaro gadget $(U \subset V)$ where $ V$ is a $G$-representation and $ j: U \subset V$ an open subscheme on which $G$ acts freely and such that the reduced complement $ \iota : Z \hookrightarrow V$ satisfies $ c:= \text{codim}_V Z > n - s $ and such that the quotient $ (U \times X) / G$ exists as a scheme. Let $ l : = \dim V$ and $ g : = \dim G$. Then one can define the equivariant higher Chow groups as 
\begin{center}
$\chow^G_s(X, t) : = \chow_{s + l - g} ((U \times X) / G, t).$
\end{center}
For the associated stack $\euX := [X/G]$ we define the (higher) Chow groups of $\euX$ as
\begin{center}
$\chow_s (\euX, t) : = \chow^G_{s+g}(X, t) = \chow_{s + l} ((U \times X) / G, t).$
\end{center}
Note that from this definition of we automatically have
$\chow_s (\euX, t) = 0 \text{ for } s > \dim (\euX) = n - g$. One checks this definition is well defined in the same way checks the definition of equivariant Chow groups of Edidin-Graham is well defined \cite{EG}. 

The next proposition is probably well known and its proof follows a standard way of arguing \cite{KaRa}[12.4] and \cite{RicharzScholbach}[Thm. 2.2.10]. We include it here because it will serve as a warm up for \cref{Theorem:Connectivity_of_mapping_spaces_in_DM}.

\begin{Proposition}
\label{Proposition:Borel_Moore_homology_comparison_arbitrary_coefficients}
Suppose $\euX = [X/G]$ and the integers $s, t \in \mathbf{Z}$ as in the discussion above. Let $ f: \euX \rightarrow B$ be the structure map. We have the following equivalence
    \begin{center}
        $\pi_0 \map_{\DM(\euX, \Lambda)}  (1_\euX (s)[2s+t], f^!1_B) \simeq \chow_s(\euX, t)_\Lambda.$
    \end{center}
\end{Proposition}

\begin{proof}
First we choose an embedding $G \hookrightarrow GL_r$. Fix a Totaro gadget $ (U,V)$ with
\begin{center}
    $\text{codim}_V(Z) > n - s + r^2 -g .$
\end{center}
Let $p: \euV : = [V/G] \rightarrow BG$ the induced vector bundle over $BG$ and $ p_\euX$ is base change to $\euX$. Then by homotopy invariance we have that $ p^* $ is fully-faithful. Thus we have an equivalence 
\begin{equation}
\map_{\DM(\euX)} (1 (s)[2s+t], f^!1_B) \simeq \map_{\DM(\euV \times_{BG} \euX)} (1 (s)[2s+t], p_\euX^*f^!1_B).
\end{equation}
Let $ \bar{j} : \euU := [U/G] \rightarrow \euV$ be the map induced by the open immersion $ j: U \rightarrow V$ and $\bar{j}_\euX$ its base change to $\euX$.  We claim that induced morphism 
\begin{equation}
\map_{\DM(\euV \times_{BG} \euX)} (1 (s)[2s+t], p_\euX^*f^!1_B) \overset{\bar{j}^*_\euX}{\rightarrow} \map_{\DM(\euU \times_{BG} \euX)} (1 (s)[2s+t], \bar{j}_\euX^* p_\euX^*f^!1_B)
\end{equation}
is an equivalence. 

Since $ p_\euX$ is smooth, separated and representable, we map apply the purity isomorphism $ p_\euX ^! \simeq p^* (l) [2l]$ which gives
\begin{center}
$\map_{\DM(\euV \times_{BG} \euX)} (1 (s+l)[2s+2l+t], p_\euX ^!f^!1_B) \overset{\bar{j}_\euX}{\rightarrow} \map_{\DM(\euU \times_{BG} \euX)} (1 (s+l)[2s+2l+t], \bar{j}_\euX^! p_\euX^!f^!1_B).$
\end{center}
Writing $ \pi:\euV \times_{BG} \euX \rightarrow B$ and $\sigma : \euU \times_{BG} \euX \rightarrow B$ for the structure maps we can rewrite this is as
\begin{equation}
\label{Equation:Borel_Moore_homology_comparison_1}
\map_{\DM(\euV \times_{BG} \euX)} (1 (s+l)[2s+2l+t], \pi^! 1_B) \overset{\bar{j}_\euX}{\rightarrow} \map_{\DM(\euU \times_{BG} \euX)} (1 (s+l)[2s+2l+t], \sigma^! 1_B).
\end{equation}
To see that \ref{Equation:Borel_Moore_homology_comparison_1} is an equivalence via the localization triangle
\begin{center}
 $ i_* i^! \rightarrow \id \rightarrow j_* j^!, $
\end{center}
we are reduced to showing that
\begin{center}
$\pi_0 \map_{\DM(\euZ \times_{BG} \euX)} (1 (s+l)[r], \bar{\iota}^!\pi^! 1_B) = 0 $
\end{center}
for all $r \in \mathbf{Z}$.

As in \cite{chowdhury2021motivic}[Rem. 2.3.7] we may find a Nis-loc atlas:
\begin{center}
    $ W \rightarrow (X \times Z) \times^G GL_r \rightarrow \euZ \times_{BG} \euX$ 
\end{center}
where $W$ is a scheme and the first arrow is an \'etale surjection. 

Since we can compute $\DM(\euZ \times_{BG} \euX)$ along \v{C}ech nerves of Nis-loc atlases it will be enough to show that $\pi_0 \map_{\DM(\euZ \times_{BG} \euX)} (1 (s+l)[r], \bar{\iota}^!\pi^! 1_B)$ vanishes on the $!$-restriction to each term 
\begin{center}
    $W^a : = \underbrace{ W \times_{\euZ \times_{BG} \euX} \cdots \times_{\euZ \times_{BG} \euX} W}_{a}$ 
\end{center}
for $ a \geq 0$ in the \v{C}ech nerve of $ W \rightarrow \euZ \times_{BG} \euX$. 

Writing $ \eta^!_a : \DM (\euZ \times_{BG} \euX) \rightarrow \DM(W^a)$ for the $!$-restriction in the \v{C}ech nerve, the purity isomorphism gives
\begin{center}
    $\eta^!_a \simeq \eta^*_a (a\gamma)[2a\gamma]$
\end{center}
where $\gamma$ is the relative dimension of the Nis-loc atlas $ W \rightarrow \euZ \times_{BG} \euX$. We also have $\eta_a^! \bar{\iota}^! \pi^! \simeq h_a^!$ where $h_a :W^a \rightarrow B$ the structure map. We are reduced to showing
\begin{center}
    $\pi_0 \map_{\DM(W^a )} (1 (s+l+a\gamma)[r+2a\gamma], h_a^! 1_B) = 0 $
\end{center}
for all $a \geq 0$.

But now we are in the realm of finite type schemes over a field and we know that
\begin{center}
$\pi_0 \map_{\DM(W^a )} (1 (s+l+a\gamma)[r+2a\gamma], h_a^! 1_B) = \chow_{s+l+a\gamma}(W^a, r-s-l)$
\end{center}
and by our choice of Totaro gadget we have that
\begin{center}
    $l + s > n + \dim(Z) +r^2 - g $
\end{center} 
and thus $   l + s + a\gamma > n+ \dim(Z) + a\gamma$. In particular because the Chow groups vanish, we conclude that 
\begin{center}
$\pi_0 \map_{\DM(W^a )} (1 (s+l+a\gamma)[r+2a\gamma], h_a^! 1_B) = 0$
\end{center}
which is what we wanted to show.
\end{proof}

The next theorem will be important in establishing the weight structure on $\DM(\euX, \Lambda)$, we will use the symbol $\Map$ to refer to the \emph{mapping spectra} as opposed to the symbol $\map$ which denotes the \emph{mapping space}. 

\begin{Theorem}
\label{Theorem:Connectivity_of_mapping_spaces_in_DM}
Suppose that $\euS = [S/G]$ where $S$ is a finite type scheme over $B$ and $G$ is an affine algebraic group. Let $\euX$ and $\euY$ be smooth stacks over $B$ and projective over $\euS$ and $ j, m , n \in \mathbf{Z}$ and let $d_{\euY}$ be the dimension of $\euY$ over $B$. 
\begin{center}
    $ \pi_j \Map_{\DM(\euS, \Lambda)} (f_!1_\euX (m)[2m], g_!1_\euY (n) [2n]) \simeq \chow_{d_{\euY} -n +m} (\euX \times_\euS \euY, j),$
\end{center}
in particular the mapping spectrum
\begin{center}
$\Map_{\DM(\euS, \Lambda)} (f_!1_\euX (m)[2m], g_!1_\euY (n) [2n])$
\end{center}
is connective. 
\end{Theorem}

\begin{proof}
First we fix an embedding $G \hookrightarrow GL_r$ and we fix a Totaro gadget $ (U,V)$ for $G$ so that 
\begin{center}
    $ c:= \text{codim}_V (Z) > \dim(X) - \dim(S) +n - m +r^2 - g$.
\end{center}
Consider the cartesian squares
\begin{center}
	\begin{tikzpicture}
	   \node (TL) at (0 , 1.5) {$ \euV \times_{BG} \euX$};
	   \node (TR) at (1.8, 1.5) {$\euX$};
	   \node (BL) at (0, 0){$ \euV \times_{BG} \euS $};
       \node (BR) at (1.8,0){$\euS$};
       \node (TLL) at (4 , 1.5) {$ \euV \times_{BG} \euY$};
	   \node (TRR) at (5.8, 1.5) {$\euY$};
	   \node (BLL) at (4, 0){$\euV \times_{BG} \euS$};
       \node (BRR) at (5.8,0){$\euS.$};

       \draw[->] (TLL) -- (TRR) node [midway, above] {$p_\euY$};
	   \draw[->] (TLL) -- (BLL)node [midway, left] {$g_\euV$};
	   \draw[->] (TRR) -- (BRR)node [midway, right] {$g$};
        \draw[->] (BLL) -- (BRR) node [midway, below] {$p$};
	   \draw[->] (TL) -- (TR) node [midway, above] {$p_\euX$};
	   \draw[->] (TL) -- (BL)node [midway, left] {$f_\euV$};
	   \draw[->] (TR) -- (BR)node [midway, right] {$f$};
      \draw[->] (BL) -- (BR) node [midway, below] {$p$};
\end{tikzpicture}
\end{center}
Combined with base change these give the following equivalences
\begin{eqnarray}
\label{Equation:connectivity_of_mapping_spectra_1}
p^* f_!1_\euX \simeq {f_\euV}_! p_\euX^*1_\euX \simeq {f_\euV}_! 1_{\euV\times_{BG} \euX}& & p^* g_!1_\euY \simeq {g_\euV}_! p_\euY^*1_\euY \simeq {g_\euV}_! 1_{\euV \times_{BG} \euY}.
\end{eqnarray}
Since $p$ is a vector bundle over $\euS$ it follows by homotopy invariance that $ p^*$ is fully-faithful which when combined with
\ref{Equation:connectivity_of_mapping_spectra_1} gives an equivalence
\begin{equation}
\map_{\DM(\euS)} (f_!1_\euX (m)[2m+j], g_!1_\euY (n) [2n]) \overset{p^*}{\simeq} \map_{\DM(\euV \times_{BG} \euS)} ({f_\euV}_!1 (m)[2m+j], {g_\euV}_!1(n) [2n]).
\end{equation}
Let $ \bar{j}: \euU \times_{BG} \euS \rightarrow \euV \times_{BG} \euS$ be the open immersion induced by $U \subset V$. Via the cartesian diagrams 
\begin{center}
	\begin{tikzpicture}
	   \node (TL) at (0 , 1.5) {$ \euU \times_{BG} \euX$};
	   \node (TR) at (1.8, 1.5) {$\euX$};
	   \node (BL) at (0, 0){$ \euU \times_{BG} \euS$};
       \node (BR) at (1.8,0){$\euS$};
       \node (TLL) at (4 , 1.5) {$ \euU \times_{BG} \euY$};
	   \node (TRR) at (5.8, 1.5) {$\euY$};
	   \node (BLL) at (4, 0){$ \euU \times_{BG} \euS$};
       \node (BRR) at (5.8,0){$\euS$};

       \draw[->] (TLL) -- (TRR) node [midway, above] {$\bar{j}_\euY$};
	   \draw[->] (TLL) -- (BLL)node [midway, left] {$g_\euU$};
	   \draw[->] (TRR) -- (BRR)node [midway, right] {$g$};
        \draw[->] (BLL) -- (BRR) node [midway, below] {$\bar{j}$};
	   \draw[->] (TL) -- (TR) node [midway, above] {$\bar{j}_\euX$};
	   \draw[->] (TL) -- (BL)node [midway, left] {$f_\euU$};
	   \draw[->] (TR) -- (BR)node [midway, right] {$f$};
      \draw[->] (BL) -- (BR) node [midway, below] {$\bar{j}$};
\end{tikzpicture}
\end{center}
we get the equivalences
\begin{eqnarray}
    \bar{j}^* f_!1_\euX \simeq {f_\euU}_! \bar{j}_\euX ^*1_\euX \simeq {f_\euU}_! 1_{\euU\times_{BG} \euX}& & \bar{j}^* g_!1_\euY \simeq {g_\euU}_! \bar{j}_\euY ^*1_\euY \simeq {g_\euU}_! 1_{\euU\times_{BG} \euY}.
\end{eqnarray}
Composing $\bar{j}^*$ with $p^*$ gives a map 
\begin{equation}
\label{Equation:connectivity_of_mapping_spectra_2}
\map_{\DM(\euS)} (f_!1_\euX (m)[2m +j], g_!1_\euY (n) [2n]) \overset{\bar{j}^*p^*}{\rightarrow} \map_{\DM(\euU \times_{BG} \euS)} ({f_\euU}_!1 (m)[2m+j], {g_\euU}_!1(n) [2n]).
\end{equation}
we claim that \cref{Equation:connectivity_of_mapping_spectra_2} is an equivalence. 

From considering the localization triangle 
\begin{center}
$ \bar{\iota}_* \bar{\iota}^! \rightarrow \id \rightarrow \bar{j}_* \bar{j}^!, $
\end{center}
we simply need to show that
\begin{center}
$ \map_{\DM(\euV \times_{BG} \euS)} ({f_\euV}_!1 (m)[2m+j], \bar{\iota}_*\bar{\iota}^!{g_\euV}_!1(n) [2n]) \simeq 0.$
\end{center}
Thus it will be enough to prove the following:
\begin{Claim}
\label{Claim:Connectivity of mapping spectra techincal claim}
\begin{equation}
\label{Equation:connectivity_of_mapping_spectra_3}
\pi_0 \map_{\DM(\euZ \times_{BG} \euS)} ({f_\euZ}_!1, \bar{\iota}^!{g_V}_!1 (n-m) [r]) \simeq 0
\end{equation}
for all $ r \in \mathbf{Z}$.  
\end{Claim}
Via standard arguments we can reduce to the situation where $Z$ is regular, in which case we have by absolute purity $ \iota^* \simeq \iota^!(c)[2c]$. 

Via the diagram of cartesian square

\begin{center}
	\begin{tikzpicture}
	   \node (TL) at (0 , 1.5) {$ \euZ \times_{BG} \euY$};
	   \node (TM) at (2, 1.5) {$\euV \times_{BG} \euY$};
        \node (TR) at (4, 1.5) {$\euY$};
	   \node (BL) at (0, 0) {$ \euZ \times_{BG} \euS$};
       \node (BM) at (2,0) {$\euV \times_{BG} \euS$};
       \node (BR) at (4, 0) {$\euS$};

       \draw[->] (TL) -- (TM) node [midway, above] {$\bar{\iota}_\euY$};
	   \draw[->] (TM) -- (TR) node [midway, above] {$p_\euY$};
	   \draw[->] (BL) -- (BM) node [midway, below] {$\bar{\iota}$};
        \draw[->] (BM) -- (BR) node [midway, below] {$p$};
	   \draw[->] (TM) -- (BM) node [midway, right] {$g_\euV$};
	   \draw[->] (TL) -- (BL)node [midway, left] {$g_\euZ$};
	   \draw[->] (TR) -- (BR)node [midway, right] {$g$};
\end{tikzpicture}
\end{center}
we have the base change equivalence
\begin{center}
$ \bar{\iota}^! {g_\euV}_! \simeq {g_\euZ}_!\bar{\iota}^!_\euY.$
\end{center}
which when combined with absolute purity for $\iota$ allows us to rewrite the \ref{Equation:connectivity_of_mapping_spectra_3} as
\begin{equation}
\label{Equation:connectivity_of_mapping_spectra_4}
     \pi_0 \map_{\DM(\euZ \times_{BG} \euS)} ({f_\euZ}_!1, {g_\euZ}_!1 (n-m-c) [r -2c]) \simeq 0.
\end{equation}

Since $\euZ \times_{BG} \euS \simeq [ (Z \times S)/G]$, by \cite{chowdhury2021motivic}[Rem. 2.3.7] we have a Nis-loc atlas $ W \rightarrow \euZ \times_{BG} \euS$. Hence we can compute $DM(\euZ \times_{BG} \euS)$ via the \v{C}ech nerve
\begin{center}
    $\cdots \triplerightarrow W^2: = W \times_{\euZ \times_{BG} \euS} W \rightrightarrows W \rightarrow \euZ \times_{BG} \euS.$
\end{center}
Thus it is enough to show \ref{Equation:connectivity_of_mapping_spectra_4} on the restriction to each $\DM(W^q)$. We write $ \pi^*_q$ for the restiction $\DM(\euZ \times_{BG} \euS)\rightarrow \DM(W^q)$. Then we can write the mapping space
\begin{equation*}
    \map_{\DM(W^q)} (\pi_q^* {f_\euZ}_!1, \pi_q^*  {g_\euZ}_!1 (n-m-c) [r -2c]),
\end{equation*}
as
\begin{equation*}
    \map_{\DM(W^q)} ({{f_\euZ}_q }_!1, {{g_\euZ}_q}_!1 (n-m-c) [r -2c]),
\end{equation*}
where ${f_{\euZ}}_q : W_\euX ^q \rightarrow  W^q$ (resp. ${g_{\euZ}}_q$) is the base change of $ f_\euZ$ along the map $ \pi_q: W^q \rightarrow \euZ \times_{BG} \euS$. 

We must show 
\begin{center}
 $\pi_0 \map_{\DM(W^q)} ({{f_\euZ}_q }_!1, {{g_\euZ}_q}_!1 (n-m-c) [r -2c]) = 0$.
\end{center}
Via \cite{Jin}[Lem. 2.37] we have the equivalence
\begin{center}
 $\pi_0 \map_{\DM(W^q)} ({{f_\euZ}_q }_!1, {{g_\euZ}_q}_!1 (n-m-c) [r -2c]) \simeq H^{BM}_{2d_q -r +2c,d_q -n + m +c }  (W^q_\euX \times_{W^q} W^q_\euY)$
\end{center}
where $d_q := \dim(Y) + \dim(Z) + r^2 - g + q\gamma$. Now comparing with the Chow groups
\begin{center}
    $H^{BM}_{2d_q -r +2c,d -n + m +c }  (W^q_\euX \times_{W^q} W^q_\euY) \simeq \chow_{d_q - n + m+ c} (W^q_\euX \times_{W^q} W^q_\euY, -r+2n-2m-c)$
\end{center}
we see that
\begin{equation*}
    d_q - n +m +c > \dim (W^q_\euX \times_{W^q} W^q_\euY).
\end{equation*}
thus these groups vanish proving \cref{Claim:Connectivity of mapping spectra techincal claim}. 

To see how the main result follows from \cref{Claim:Connectivity of mapping spectra techincal claim}, note that it's consequence is the equivalence
\begin{equation}
\label{Equation:Connectivity jin}
    \map_{\DM(\euS)} (f_!1_\euX (m)[2m +j], g_!1_\euY (n) [2n]) \overset{\bar{j}^*p^*}{\rightarrow} \map_{\DM(\euU \times_{BG} \euS)} ({f_\euU}_!1 (m)[2m+j], {g_\euU}_!1(n) [2n]).
\end{equation}
Now we note that the stack $\euU \times_{BG} \euX \times_\euS \euY$ is equivalent to $ \euU \times_{BG} \euX \times_{\euU \times_{BG} \euS} \euU \times_{BG} \euY$, which means that it is a scheme. Following the arguments of \cite{Jin}[Lem. 2.37] we can identify the right hand side of \cref{Equation:Connectivity jin} with 
\begin{center}
  $ \pi_0  \map_{\DM(\euU \times_{BG} \euX \times_\euS \euY)} (1(l + \dim(\euY)  + m -n)[2m-2n+2l +2\dim(\euY) +j], a^! 1_B)  $
\end{center}
via base change and purity, where $a:\euU \times_{BG} \euX \times_\euS \euY) \rightarrow B$ is the structure morphism.  We now see by \cref{Proposition:Borel_Moore_homology_comparison_arbitrary_coefficients} that this is just 
\begin{center}
   $ \chow_{l +d_\euY - n + m} (\euU \times_{BG} \euX \times_\euS \euY, j) \simeq \chow _{d_\euY - n + m}(\euX \times_\euS \euY , j)$
\end{center}
which is what we wanted to show.
\end{proof}

\section{Weight Structures}
\label{Section:Weight Structure}

We first remind the reader of the definition of a weight structure on the stable $\infty$-category. 

\begin{Definition}
    A weight structure on a stable $\infty$-category $\mathscr{C}$, is the datat of two retract closed subcategories $(\mathscr{C}_{w \geq 0} , \mathscr{C}_{w \leq 0})$ such that:
    \begin{enumerate}
    \item $\Sigma \mathscr{C}_{w \geq 0}  \subset \mathscr{C}_{w \geq 0} ,\  \Omega \mathscr{C}_{w \leq 0}  \subset \mathscr{C}_{w \leq 0 }$. We write
    \begin{center}
    $\mathscr{C}_{w \geq n} : = \Sigma^n \mathscr{C}_{w \geq 0}, \mathscr{C}_{w \leq n} : = \Omega^n \mathscr{C}_{w \leq 0}$
    \end{center}
\item If $ x \in \mathscr{C}_{w \leq 0}, y \in \mathscr{C}_{w \geq 1}$ then
\begin{center}
$ \pi_0 \map (x, y ) \simeq 0$.
\end{center}
\item For any $ x \in \mathscr{C}$ we have a cofiber sequence 
\begin{center}
$ x_{\leq 0} \rightarrow x \rightarrow x_{\geq 1}$.
\end{center}
with $x_{\leq 0}\in \mathscr{C}_{w \leq 0}$ and $ x_{\geq 1}  \in \mathscr{C}_{w \geq 1}$,  called the weight truncations of $x$. 
\end{enumerate}
We say that a weight structure is bounded if 
\begin{center}
    $\mathscr{C} = \bigcup_{n \in \mathbf{Z}} (\mathscr{C}_{w \geq -n} \cap \mathscr{C}_{w \leq n})$.
\end{center}
We also define the weight heart of weight structure to be 
\begin{center}
    $\mathscr{C}^{\heartsuit_w} : = \mathscr{C}_{w \geq 0} \cap \mathscr{C}_{w \leq 0}$.
\end{center}
\end{Definition}

Next we state a theorem due to Bondarko \cite{bondarko_2010}[4.3.2.II], but see also H\'ebert \cite{hébert_2011}[Thm 1.9]. (see also \cite{ElmantoSosnilo}[Rem. 2.2.6] for the $\infty$-categorical version, which we state here)
\begin{Theorem}
\label{Theorem:Existence_of_abstract_weight_structure}
    (Bondarko) Let $\mathscr{C}$ be a stable $\infty$-category. Assume we are given a full subcategory $ \mathscr{B} \subset \mathscr{C}$ such that
    \begin{enumerate}
\item $\mathscr{B}$ generates $\mathscr{C}$ under finite limits, finite colimits and retracts.
\item $\mathscr{B}$ has connective mapping spectra. 
\end{enumerate}
Then we may define the following subcategories
\begin{center}
$\mathscr{C}_{w \geq 0} = \{ \text{retracts of finite colimits of objects of } \mathscr{B} \}$
\end{center}
and 
\begin{center}
$\mathscr{C}_{w \leq 0} = \{ \text{retracts of finite limits of objects of } \mathscr{B} \}.$
\end{center}
These subcategories give a bounded weight structure on $\mathscr{C}$ whose heart is the minimal retract-closed additive subcategory containing  $\mathscr{B}$. 
\end{Theorem}

\begin{Theorem} 
\label{Theorem:Weight_structure_for_global_quotients}
Let $\euX = [X/G]$ where $X$ is a quasi-projective scheme over a field $k$ of characteristic $0$ and $G$ is an affine algebraic group.

\begin{enumerate}
    \item The $\infty$-cateegory $\DM_{\gm}(\euX, \Lambda)$ admits a bounded weight structure, with 
    \begin{center}
        $\DM_{\gm}(\euX, \Lambda)^{\heartsuit_w} \simeq \CM(\euX, \Lambda)$.
    \end{center}

    \item The $\infty$-category $\Ind \DM_{\gm} (\euX, \Lambda)$ admits a weight structure which restricts to the weight structure on $\DM_{\gm} (\euX, \Lambda)$ constructed in $(1)$. 
\end{enumerate}
\end{Theorem}

\begin{proof}
For the first claim simply have to verify the conditions \cref{Theorem:Existence_of_abstract_weight_structure}. In the notation of that theorem we take
\begin{eqnarray*}
    \mathscr{C} := \DM_{\gm} (\euX, \Lambda) & & \mathscr{B} : = \CM(\euX, \Lambda).
\end{eqnarray*}
Condition $(1)$ follows from \cref{Theorem:Generation_of_DM_gm_by_chow}
and condition  $(2)$ follows from \cref{Theorem:Connectivity_of_mapping_spaces_in_DM}. For the second claim we can use \cite{BondarkoIvanov} [Prop. 1.4.2 (9)] to finish the argument.
\end{proof}

\section{Equivariant Motives}
\label{Section:Equivariant Motives}

 In this final section we identify the homotopy category of Chow motives $\h \CM(\euS, \Lambda)$ with the natural generalization of both Laterveer's category of $G$-equivariant Chow motives  \cite{Laterveer} as well as Corti and Hanamura's category of Chow motives over a general base \cite{CortiHanamura} when $\euS$ is a global quotient stack. That is to say when $\euS$ is $BG$ our identification will show that $\h \CM(\euS, \mathbf{Q})$ is equivalent to Laterveer's original category.

 Let $\euS = [S/G]$ where $S$ is quasi-projective over $B : = \spec(k)$ and $G$ is an affine algebraic group over $B$. Suppose that $\euX, \euY$ are smooth over $B$ and projective over $\euS$. Then following \cite{CortiHanamura} we define the set of correspondences of degree $r$ between $\euX$ and $\euY$ as follows:  Let $ \euY = \coprod_{i} \euY_i$ with $\euY_i$ irreducible components then 
\begin{equation*}
    \Corr_r (\euX, \euY) : = \bigoplus_{i} \chow_{\dim \euY_{i} +r} ( \euX \times_\euS \euY_i)_\Lambda. 
\end{equation*}
We can construct a composition of correspondences
\begin{equation*}
    \circ: \Corr_r (\euX, \euY) \otimes \Corr_s (\euY, \euZ) \rightarrow \Corr_{r+s} (\euX, \euZ)
\end{equation*}
by considering the diagram 
\begin{center}
	\begin{tikzpicture}
		\node (TL) at (0 , 1.5) {$ \euX \times_{\euS} \euY \times_{\euS} \euZ$};
		\node (TR) at (5, 1.5) {$ \euX \times_{\euS} \euY \times  \euY \times_{\euS} \euZ$};
		\node (BL) at (0, 0){$\euY$};
        \node (BR) at (5,0){$ \euY \times \euY$};
		
		\draw[->] (TL) -- (TR) ;
		\draw[->] (TL) -- (BL);
		\draw[->] (TR) -- (BR);
        \draw[->] (BL) -- (BR) node [midway, below] {$\delta$};
	\end{tikzpicture}
\end{center}
which  allows us to define
\begin{equation*}
    \alpha \circ \beta := {p_{\euX \euZ}}_*( \delta^! (\alpha \times \beta)),
\end{equation*}
where $\delta: \euY \rightarrow \euY \times \euY$ and $p_{\euX, \euZ} : \euX \times_\euS \euY \times_\euS \euZ \rightarrow \euX \times_\euS \euZ$ the projection. We note that identity $ \id \in \Corr_0 (\euX, \euX)$ is given by the scheme theoretic image of $ \Delta : \euX \rightarrow \euX \times_{\euS} \euX$. 
\begin{Definition}
\label{Definition:chow_motives}
Let $\LM(\euS, \Lambda)$ denote the classical category Chow motives. The objects of this category are triples
\begin{center}
$(\euX, p, m)$
\end{center}
where $\euX$ is  smooth and projective over $\euS$, $p$ is an idempotent in $\Corr_0(\euX,\euX)$ and $m \in \mathbf{Z}$. The morphism sets are defined as
\begin{center}
$\hom ((\euX, p, m), (\euY, q, n)) : = q \circ  \Corr_{m-n} (\euX,\euY) \circ p \subseteq \Corr_{m-n} (\euX,\euY).$
\end{center}
\end{Definition}

\begin{Example}
\label{Example:Laterveer motives}
When $\euS = BG$, the category $\LM (\euS, \mathbf{Q})$ is equivalent to Laterveer's category of $G$-equivariant motives \cite{Laterveer}. In particular, to get the equivalence one must re-index because in loc. cit. Chow cohomology is used and in our situation since we are working over a not necessarily smooth stack $\euS$ we must use Chow homology.
\end{Example}

\begin{Lemma}
The category $\LM(\euS, \Lambda)$ is an additive, idempotent complete and symmetric monoidal where the tensor product is defined as 
\begin{center}
$(\euX, p, m) \otimes (\euY, q, n) : = (\euX \times_\euS \euY, p \times q, m+n).$
\end{center}
\end{Lemma}

\begin{proof}
The proof follows by combining the arguments of \cite{stacks-project}[0FGF], \cite{stacks-project}[0FGB] and \cite{stacks-project}[0FGC].  
\end{proof}

We will now construct a functor
\begin{center}
$ F: \LM(\euS, \Lambda) \rightarrow \h \CM(\euS, \Lambda).$
\end{center}
Since $\LM(\euS, \Lambda)$ is the idempotent completion of the full sub-category $\LM(\euS, \Lambda)'$ spanned by pairs $(\euX, \id, m)$,  and $ \h \CM(\euS, \Lambda)$ is idempotent complete it is enough to construct a functor 
\begin{center}
$F' : \LM(\euS, \Lambda)' \rightarrow \h \CM(\euS, \Lambda).$
\end{center}
On objects this functor sends $ (\euX, m)$ to $f_{!}1_{\euX} (m)[2m] \in \h \CM(\euS, \Lambda)$. To describe what this functor does on morphisms we consider the isomorphism 
\begin{equation}
\label{Equation:naturaly_for_cycles}
\epsilon_{\euX,\euY}: \pi_0 \map_{\CM(\euS, \Lambda)} (f_! 1_\euX (m)[2m], g_{!} 1_\euY (n)[2n]) \simeq \chow_{d_{\euY} -n+m} (\euX \times_{\euS} \euY). 
\end{equation}
constructed in the proof of \cref{Theorem:Connectivity_of_mapping_spaces_in_DM}. We define for $\alpha \in \Corr_{m-n} (\euX, \euY)$ 
\begin{equation*}
    F' (\alpha) := \epsilon_{\euX, \euY}^{-1} (\alpha). 
\end{equation*}
In order to verify that $F'$ is a functor we need to check that $ F'(\alpha \circ \beta) = F'(\alpha) \circ F'(\beta)$. That is $F'$ is natural with respect to composition of cycles. The following is a version of \cite{Jin}[Prop. 2.39] for stacks. 

\begin{Proposition}
\label{Proposition:Naturality_for_composition_of_cycles}
Let $f: \euX \rightarrow \euS , f:\euY \rightarrow \euS, h:\euZ \rightarrow \euS$ be projective morphisms, where $\euX, \euY$ and $\euZ$ are smooth over $B$ and let $\alpha: f_!1_\euX \to g_! 1_\euY$ and $\beta: g_! 1_\euY \to h_! 1_\euZ$. Consider the cartesian diagram
\begin{center}
	\begin{tikzpicture}
		\node (TL) at (0 , 1.5) {$ \euX \times_{\euS} \euY \times_{\euS} \euZ$};
		\node (TR) at (5, 1.5) {$ \euX \times_{\euS} \euY \times  \euY \times_{\euS} \euZ$};
		\node (BL) at (0, 0){$\euY$};
        \node (BR) at (5,0){$ \euY \times \euY$};
		
		\draw[->] (TL) -- (TR) ;
		\draw[->] (TL) -- (BL);
		\draw[->] (TR) -- (BR)node [midway, right] {$f$};
        \draw[->] (BL) -- (BR) node [midway, below] {$\delta$};
	\end{tikzpicture}
\end{center}
Then we have the following equality
\begin{center}
$\epsilon_{\euX,\euZ}(\beta \circ \alpha) = \pi_* \delta^! (\epsilon_{\euX,\euY}(\alpha) \circ \epsilon_{\euY,\euZ} (\beta))$
\end{center}
 where $\pi:  \euX \times_{\euS} \euY \times_{\euS} \euZ \rightarrow  \euX \times_{\euS} \euZ $ is the projection.
\end{Proposition}

\begin{proof}
We fix a Totaro gadget $ U \subset V$ and let $ q: U/G \rightarrow B$. Then as in the proof of \cref{Theorem:Connectivity_of_mapping_spaces_in_DM}  the functor $ q^* = \bar{j}^* p^*$ induces natural isomorphisms
\begin{eqnarray*}
    \pi_0 \map_{\DM(\euS, \Lambda)} (f_! 1, g_! 1) \overset{q^*}{\rightarrow} \pi_0 \map_{\DM(\euU \times_{BG} \euS, \Lambda)} ({f_{\euU}}_!1, {g_{\euU}}_! 1 ) \\
    \pi_0 \map_{\DM(\euS, \Lambda)} (g_! 1, h_! 1) \overset{q^*}{\rightarrow} \pi_0 \map_{\DM(\euU \times_{BG} \euS, \Lambda)} ({g_{\euU}}_!1, {h_{\euU}}_! 1 ).
\end{eqnarray*}
The result now follows from \cite{Jin}[Prop. 2.39], applied to 
\begin{eqnarray*}
    f_{\euU}: \euU \times_{BG} \euX \rightarrow \euU \times_{BG} \euS,\\
    g_{\euU}: \euU \times_{BG} \euY \rightarrow \euU \times_{BG} \euS, \\
    h_{\euU}: \euU \times_{BG} \euZ \rightarrow \euU \times_{BG} \euS.
\end{eqnarray*} 
\end{proof}

\begin{Corollary}
    \label{Corollary:F'_is_a_functor}
    The map $F' : \LM(\euS, \Lambda)' \rightarrow \h \CM(\euS, \Lambda)$ is a functor.
\end{Corollary}

\begin{proof}
    This follows directly from \cref{Proposition:Naturality_for_composition_of_cycles} together with \cite{Jin}[Props. 3.11, 3.15, 3.16]. 
\end{proof}

Now by the universal property of idempotent completion we get a well defined functor
\begin{equation} 
\label{Equation:weight_heart_identification}
 F: \LM(\euS, \Lambda) \rightarrow \h \CM(\euS, \Lambda)
\end{equation}

\begin{Theorem}
\label{Theorem:Identification of chow motives}
The functor $F$,  \ref{Equation:weight_heart_identification}, is an equivalence of categories
\begin{center}
    $F: \LM(\euS, \Lambda) \overset{\simeq}{\rightarrow} \h \CM(\euS, \Lambda)$
\end{center}
\end{Theorem}

\begin{proof}
Fully-faithfullness is clear. To see that it is essentially surjective simply note that $F$ is an additive functor and every generator of $\h \CM(BG)$ is contained in its essential image.
\end{proof}

\begin{Corollary}
\label{Corollary:Identification of weight heart}
Let $\euX$ be a Nis-loc stack over a field of characteristic $0$. Then the homotopy category of the heart of the  weight structure constructed in \cref{Theorem:Weight_structure_for_global_quotients} for $\DM_{\gm}(\euX, \Lambda)$ can be identified with $\LM(\euX, \Lambda)$. That is
\begin{equation*}
    \LM(\euX, \Lambda) \simeq \h \DM_{\gm} (\euX, \Lambda)^{\heartsuit_w} .
\end{equation*}
\end{Corollary}

\begin{Corollary}
    \label{Corollary:Comparison with Laterveer's category}
For an algebraic group $G$, the category $ h\CM(BG, \mathbf{Q})$ is equivalent to the category of $G$-equivariant chow motives of Laterveer. 
\end{Corollary}

\begin{proof}
This is just \cref{Example:Laterveer motives} combined with \cref{Theorem:Identification of chow motives}. 
\end{proof}

\bibliographystyle{amsalpha}
\newcommand{\etalchar}[1]{$^{#1}$}
\providecommand{\bysame}{\leavevmode\hbox to3em{\hrulefill}\thinspace}
\providecommand{\MR}{\relax\ifhmode\unskip\space\fi MR }
\providecommand{\MRhref}[2]{%
  \href{http://www.ams.org/mathscinet-getitem?mr=#1}{#2}
}
\providecommand{\href}[2]{#2}

\end{document}